%% file: Pohl_dynamics_numbers.tex
\documentclass{amsart}

\usepackage[colorlinks,breaklinks]{hyperref}
\usepackage{graphics}
\allowdisplaybreaks

\input{Pohl_praeambel}

\input{Pohl_ap_makros}

\begin{document}

\title[Symbolic dynamics and the automorphic Laplacian]{Symbolic dynamics, automorphic functions, and Selberg zeta functions with unitary representations}
\author[A.\@ Pohl]{Anke D.\@ Pohl}
\address{Mathematisches Institut, Georg-August-Universit\"at G\"ottingen,  Bunsenstr. 3-5, 37073 G\"ottingen}
\curraddr{Max Planck Institute for Mathematics, Vivatsgasse 7,  53119 Bonn, Germany}
\email{pohl@mpim-bonn.mpg.de}
\begin{abstract} 
Using Hecke triangle surfaces of finite and infinite area as examples, we present techniques for thermodynamic formalism approaches to Selberg zeta functions with unitary finite-dimensional representations $(V,\chi)$ for hyperbolic surfaces (orbifolds) $\Gamma\backslash\h$ as well as transfer operator techniques to develop period-like functions for $(\Gamma,\chi)$-automorphic cusp forms. This leads to several natural conjectures. We further show how to extend these results to the billiard flow on the underlying triangle surfaces, and study the convergence of transfer operators along sequences of Hecke triangle groups.
\end{abstract}

\maketitle

\section{Introduction}

The intimate relation between the geometric and spectral properties of Riemannian locally symmetric spaces is of utmost interest in various areas, including dynamical systems, spectral theory, harmonic analysis, representation theory, number theory and quantum chaos, and contributes to their cross-fertilization. It is one reason for the increasing number of competing or complementary approaches via dynamical systems and ergodic theory on the one hand, and, e.g., harmonic analysis and analytic number theory on the other hand (see, for example,  \cite{Duke, ELMV} or \cite{Lindenstrauss_aque,Soundararajan,Watson}). Up to date, the full extent of this relation, its consequences and the properties of several entities on the geometric and the spectral side remain an active area of research. 

Over the last three decades, mostly for hyperbolic surfaces, an approach complementary to the classical analytic and number-theoretic methods emerged within the framework of the thermodynamic formalism of statistical mechanics: transfer operator techniques. These techniques focus on the dynamics of the geodesic flows rather than on the geometry of the surfaces, which allowed to establish results, a few mentioned below, hitherto unattained by any other method, or to provide alternative proofs to known results.  

One such class of examples are classical dynamical approaches to Laplace eigenfunctions, most notably Maass cusp forms, as well as representations of Selberg zeta functions as Fredholm determinants of transfer operator families, which allow us to show their meromorphic continuations. The modular surface 
\[
\PSL_2(\Z)\backslash\h = \PSL_2(\Z)\backslash\PSL_2(\R)/\PSO(2)
\]
had been the first instance for which such transfer operator approaches could be established. For this, Mayer \cite{Mayer_thermo, Mayer_thermoPSL} considered the Gauss map 
\[
 K\colon [0,1]\setminus\Q \to [0,1]\setminus\Q,\quad x \mapsto \frac1x \mod 1,
\]
which is well-known to derive from a symbolic dynamics for the geodesic flow on the modular surface \cite{Artin,Series}. He investigated its associated transfer operator $\mc L_{K,s}$ with parameter $s\in\C$, hence the operator
\[
 \mc L_{K,s}f(x) = \sum_{n\in\N} (x+n)^{-2s} f\left(\frac{1}{x+n}\right),
\]
and found a Banach space $\mc B$ on which, for $\Rea s>1/2$, the operator $\mc L_{K,s}$ acts, is nuclear of order $0$, and the map $s\mapsto \mc L_{K,s}$ has a meromorphic extension to all of $\C$ (which we also denote by $\mc L_{K,s}$). The Selberg zeta function $Z$ for the modular surface is then given as the product of the Fredholm determinants
\begin{equation}\label{Mayer_selberg}
Z(s) = \det(1-\mc L_{K,s})\det(1+\mc L_{K,s}).
\end{equation}
This provides an alternative proof of the meromorphic extension of the Selberg zeta function. Even more, it shows that the zeros of $Z$ are determined by the eigenfunctions with eigenvalue $\pm 1$ of $\mc L_{K,s}$ in $\mc B$. The natural question whether these eigenfunctions, for $\Rea s = 1/2$, actually characterize the (even/odd) Maass cusp forms for $\PSL_2(\Z)$ could be answered affirmately. Both Lewis--Zagier \cite{Lewis_Zagier} and Chang--Mayer \cite{Chang_Mayer_transop} showed that these eigenfunctions are bijective to highly regular solutions of the functional equation
\begin{equation}\label{lewiseq}
 f(x) = f(x+1) + (x+1)^{-2s} f\Big(\frac{x}{x+1}\Big), \quad x\in\R_{>0}
\end{equation}
and satisfy in addition a certain symmetry. Without using any Selberg theory, Lewis--Zagier \cite{Lewis_Zagier} continued to show that these solutions are in bijection with the even/odd Maass cusp forms, which justifies to call them period functions (in analogy to the period polynomials in Eichler--Shimura theory). In total, this transfer operator/thermodynamic formalism approach complements Selberg theory in the following sense: it provides a characterization of Maass cusp forms as eigenfunctions of transfer operators which are constructed using only the geodesic flow on the modular surface. Thus it gives a classical dynamical characterization of the Maass cusp forms themselves, not only of their spectral parameters. Moreover, by taking Fredholm determinants, it recovers the interpretation of (some of) the zeros of the Selberg zeta function.

Such kinds of results are quite sensitive to the choice of the discretization and symbolic dynamics for the geodesic flow on the considered hyperbolic surface $\Gamma\backslash\h$. For subgroups $\Gamma$ of $\PSL_2(\Z)$ of finite index they could be shown by ``twisting'' the $\PSL_2(\Z)$-setup with the unitary representation of $\PSL_2(\Z)$ induced from the trivial one-dimensional representation of $\Gamma$ \cite{Chang_Mayer_eigen, Chang_Mayer_extension, Deitmar_Hilgert}. In this way, the symbolic dynamics, the transfer operators and the functional equation get pushed from $\PSL_2(\Z)$ to $\Gamma$, and the representation essentially serves as a bookkeeping device for the cosets of $\Gamma\backslash\PSL_2(\Z)$. An alternative symbolic dynamics for $\PSL_2(\Z)$ was used in \cite{Bruggeman_Muehlenbruch}. All other used discretizations and symbolic dynamics only led (yet) to a representation of the Selberg zeta function as a Fredholm determinant \cite{Pollicott, Morita_transfer, Fried_triangle, 
Mayer_Muehlenbruch_Stroemberg}.

In \cite{Pohl_diss, Pohl_Symdyn2d} (see also \cite{Hilgert_Pohl}) we constructed discretizations for the geodesic flow on hyperbolic spaces $\Gamma\backslash\h$ for Fuchsian groups $\Gamma$ with at least one cusp and satisfying an additional (weak) geometric condition. These discretizations/symbolic dynamics are tailor-made for the requirements of transfer operator approaches to Maass cusp forms and Selberg zeta functions. They allowed us to develop the following redefinition of the structure of these approaches: We first construct a certain discrete dynamical system on the geodesic boundary of $\h$ and a certain symbolic dynamics on a \textit{finite} alphabet (we call these systems ``slow''). Hence the transfer operators associated to these systems have only finitely many terms, and their eigenfunctions are therefore obviously characterized by finite-term functional equations. If $\Gamma$ is a lattice, then the highly regular eigenfunctions with eigenvalue $1$ (period functions) are in bijection with the 
Maass cusp forms for $\Gamma$ \cite{Moeller_Pohl, Pohl_mcf_Gamma0p, Pohl_mcf_general}. 

Thus, as in the seminal approach for $\PSL_2(\Z)$, we provide classical dynamical characterizations of the Maass cusp forms. But in contrast we do not need to make a detour to the Selberg zeta function and then hope for a finite-term functional equation to pop out of an infinite-term transfer operator. Our transfer operators provide us in an immediate and natural way with the necessary functional equations. For $\PSL_2(\Z)$ our transfer operators are
\[
 \mc L_s f(x) = f(x+1) + (x+1)^{-2s} f\Big( \frac{x}{x+1} \Big).
\]
They obviously reproduce the Lewis--Zagier functional equation \eqref{lewiseq}.

These slow systems are not uniformly expanding for which reason the finite-term transfer operators are not nuclear and cannot represent the Selberg zeta function as a Fredholm determinant. To overcome this obstacle, we apply an acceleration/induction procedure to provide a uniformly expanding discrete dynamical system and symbolic dynamics, which necessarily uses an \textit{infinite} alphabet. We call these systems ``fast''. The associated infinite-term transfer operators then represent the Selberg zeta function. Also here we recover Mayer's transfer operator. In \cite{Moeller_Pohl} this induction procedure is performed for cofinite Hecke triangle groups, in this article for $\Gamma_0(2)$ respectively the Theta group, and in \cite{Pohl_hecke_infinite} for non-cofinite Hecke triangle groups. In a forthcoming article it will be shown for all admissible $\Gamma$.

In \cite{Pohl_hecke_spectral} we used the idea of these parallel, but closely related, fast/slow systems for the billiard flow on the triangle surface underlying the Hecke triangle surfaces. We showed how specific weights of these systems allow us to accommodate Dirichlet resp.\@ Neumann boundary value conditions, and in turn to geometrically separate odd and even Maass cusp forms. These results illuminate the factorization of the Selberg zeta function in \eqref{Mayer_selberg} and explain the additional symmetries needed for the solutions of the functional equation \eqref{lewiseq}.

Another class of examples of crucial input from transfer operators are studies of resonances for hyperbolic surfaces of infinite area. For Schottky surfaces (convex cocompact surfaces with no cusps and no elliptic points) the standard Markov symbolic dynamics gives rise to transfer operators which represent the Selberg zeta function. This representation is indispensable for proving the results on the distribution of resonances in \cite{Guillope_Lin_Zworski,Naud_resonancefree} and the (numerical) investigations of their fine structure in \cite{Borthwick_numerical, Weich}. Our construction of fast/slow systems also applies to non-cofinite Hecke triangle groups \cite{Pohl_hecke_infinite} (forthcoming work will extend it to other Fuchsian groups as well). We expect that the arising transfer operators now allow us to extend these results on resonances to infinite-area surfaces with cusps. Moreover, the transfer operators arising from the slow systems lead to natural conjectures on the residues at the resonances.

In this article we show how to generalize these parallel slow/fast transfer operator techniques to accommodate an arbitrary finite-dimensional unitary representation $(V,\chi)$ of the Fuchsian group $\Gamma$ under consideration. This provides us with thermodynamic formalism approaches to the Selberg zeta functions for the automorphic Laplacian with respect to $\chi$ (Section~\ref{fast}), and leads to several conjectures on period functions for $(\Gamma,\chi)$-automorphic cusp forms as well as, for infinite-area situations, to conjectures on the residues at the resonances (Section~\ref{slow}). 

We restrict ourselves here to cofinite and non-cofinite Hecke triangle groups. However, it is obvious that these techniques apply (in the same way) to other Fuchsian groups as well. The restriction to this class of Fuchsian groups has two reasons. Besides the necessity to keep this article to a reasonable length, Hecke triangle groups form a family containing cofinite and non-cofinite Fuchsian groups as well as arithmetic and non-arithmetic lattices which, in a certain sense, deform into each other. Further, the Phillips--Sarnak conjecture \cite{ Phillips_Sarnak_cuspforms, Phillips_Sarnak_weyl, Judge, Hillairet_Judge} states that even Maass cusp forms should not exist for generic cofinite Hecke triangle groups, whereas odd Maass cusp forms are known to exist in abundance. Both, the conjecture and the results in \cite{Judge, Hillairet_Judge}, are based on deformation theory. In Section~\ref{sequences} we therefore discuss the convergence of transfer operators along sequences of Hecke triangle groups.

It is well-known that any decomposition of the representation yields a corresponding factorization of the Selberg zeta functions. In Section~\ref{symmetryreduction}, we use our results to show that this decomposition already happens at the level of transfer operators and that the factorization of the zeta functions is merely a shadow of it.

All Hecke triangle groups commute with the matrix element 
\[
\bmat{-1}{0}{0}{1} \in \PGL_2(\R).
\]
Our constructions respect this symmetry in the sense that it is inherited to the discretizations, discrete dynamical systems and transfer operators. In Section~\ref{billiard} we briefly survey how to extend our results to transfer operator approaches for the billiard flow on the triangle surfaces underlying the Hecke triangle surfaces, how it can be used to distinguish the odd and even eigenfunctions and to provide separate dynamical zeta functions for different boundary conditions.

\section{Preliminaries}

\subsection{Hyperbolic geometry}

Throughout we use the upper half plane
\[
 \h \sceq \{ z\in\C \mid \Ima z > 0\}
\]
endowed with the well-known hyperbolic Riemannian metric given by the line element
\[
 ds^2 = \frac{dz d\overline{z}}{(\Ima z)^2} 
\]
as model for the hyperbolic plane. We identify its geodesic boundary with $P^1(\R) \cong \R \cup \{\infty\}$. The group of Riemannian isometries on $\h$ is isomorphic to $\PGL_2(\R)$, whose action on $\h$ extends continuously to $P^1(\R)$. Its subgroup of orientation-preserving Riemannian isometries is then $\PSL_2(\R)$, which acts by fractional linear transformations. Hence, for $\textbmat{a}{b}{c}{d}\in \PSL_2(\R)$ and $z\in\h\cup\R$ we have
\[
\bmat{a}{b}{c}{d}.z = 
\begin{cases}
\frac{az+b}{cz+d} & \text{for $cz+d\not=0$}
\\
\infty & \text{for $cz+d=0$} 
\end{cases}
\quad\text{and}\quad
\bmat{a}{b}{c}{d}.\infty = 
\begin{cases}
\frac{a}{c} & \text{for $c\not=0$}
\\
\infty & \text{for $c=0$.}
\end{cases}
\]
Let 
\begin{equation}\label{defQ}
 Q \sceq \bmat{0}{1}{1}{0} \in \PGL_2(\R).
\end{equation}
Then 
\[
 Q.z = \frac{1}{\overline z},\qquad Q.\infty = 0,
\]
and 
\[
 \PGL_2(\R) = \PSL_2(\R) \cup Q.\PSL_2(\R).
\]
Clearly, the action of $\PGL_2(\R)$ on $\h$ induces an action on the unit tangent bundle $S\h$ of $\h$.

\subsection{Hecke triangle groups}

Let $\lambda>0$. The subgroup of $\PSL_2(\R)$ generated by the two elements
\begin{equation}\label{generators}
 S\sceq \bmat{0}{1}{-1}{0}\quad\text{and}\quad  T_\lambda \sceq \bmat{1}{\lambda}{0}{1}
\end{equation}
is called the \textit{Hecke triangle group} $\Gamma_\lambda$ with parameter $\lambda$. It is Fuchsian (i.e., discrete) if and only if $\lambda \geq  2$ or $\lambda = 2\cos \frac{\pi}{q}$ with $q\in\N_{\geq 3}$. In the following we only consider Fuchsian Hecke triangle groups. A fundamental domain for $\Gamma_\lambda$ is given by (see Figure~\ref{funddoms})
\[
 \mc F_\lambda \sceq \{ z\in\h \mid |z|>1,\ |\Rea z|<\lambda/2 \}.
\]
\begin{figure}[h]
\begin{center}
\includegraphics{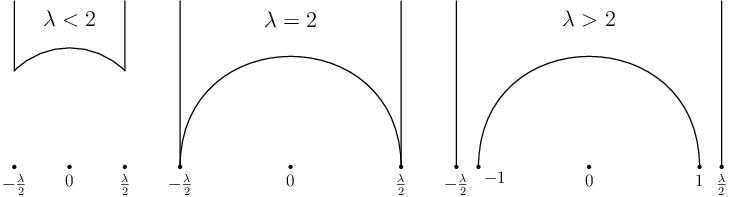}
\end{center}
\caption{Fundamental domain for $\Gamma_\lambda$.}
\label{funddoms}
\end{figure}

The side-pairings are provided by the generators \eqref{generators}: the vertical sides $\{ \Rea z = -\lambda/2\}$ and $\{ \Rea z = \lambda/2\}$ are identified via $T_\lambda$, and the bottom sides $\{ |z|=1,\, \Rea z\leq 0\}$ and $\{|z|=1,\, \Rea z\geq 0\}$ via $S$. The associated orbifold
\[
 X_\lambda \sceq \Gamma_\lambda\backslash\h
\]
is called a \textit{Hecke triangle surface}.

Among the Fuchsian Hecke triangle groups precisely those with parameter $\lambda\leq 2$ are lattices. For $\lambda = \lambda(q) =  2\cos\frac{\pi}{q}$ with $q\in\N_{\geq 3}$, $X_\lambda$ has one cusp (represented by $\infty$) and two elliptic points. In the special case $q=3$, thus $\lambda(q)=1$,  the Hecke triangle group $\Gamma_1$ is just the modular group $\PSL_2(\Z)$. The Hecke triangle group $\Gamma_2$ is called the \textit{Theta group}. It is conjugate to the projective version of $\Gamma_0(2)$. Hence $X_2$ has two cusps (represented by $\infty$ and $\lambda/2$) and one elliptic point. The groups $\Gamma_{\lambda}$ for $\lambda>2$ are non-cofinite, the orbifold $X_\lambda$ has one funnel (represented by the subset $[-\lambda/2,-1] \cup [1,\lambda/2]$ of $\R$), one cusp (represented by $\infty$) and one elliptic point.

For each Hecke triangle group $\Gamma_\lambda$ and each cuspidal point $c$ of $\Gamma_\lambda$, we let $\Stab(c,\Gamma_\lambda)$ denote the stabilizer group of $c$ in $\Gamma_\lambda$. As it is well-known, these are generated by a single (parabolic) element, say by $P_c$. We may choose $P_\infty=T_\lambda$, and for $\lambda=2$, $P_{-1} = \textbmat{2}{1}{-1}{0}$.

\subsection{Representations}

Let $\chi$ be a unitary representation of $\Gamma_\lambda$ on a finite-dimen\-si\-on\-al complex inner product space $V$.  Let $p\in\Gamma_\lambda$ be parabolic. We call $\chi$ \textit{singular in the cusp represented by} $p$ if 
\[
 \dim \ker(\chi(p) - 1) > 0,
\]
where $1$ is the identity operator on $V$. If $\chi$ is singular in at least one cusp, then we call  $\chi$ \textit{singular}. Otherwise we say $\chi$ is \textit{regular}. Further, we call
\[
 \sd(\chi) \sceq \max\{ \dim \ker(\chi(p) - 1) \mid \text{ $p\in\Gamma_\lambda$ parabolic}\}
\]
the \textit{degree of singularity} of $\chi$. We note that for $\lambda\not=2$, the degree of singularity is just
\[
 \sd(\chi) = \dim\ker(\chi(T_\lambda)-1),
\]
whereas for $\lambda=2$ it is
\[
 \sd(\chi) = \max\big\{ \dim\ker(\chi(T_2)-1), \dim\ker(\chi(P_{-1})-1) \big\}.
\]

\subsection{Automorphic functions, cusp forms, and resonances}
A function $f\colon \h \to V$ is called \textit{$(\Gamma_\lambda,\chi)$-automorphic} if 
\[
f(\gamma.z) = \chi(\gamma) f(z)
\]
for all $z\in\h$, $\gamma\in\Gamma_\lambda$. Let $C^\infty(X_\lambda;V;\chi)$ be the space of smooth ($C^\infty$) $(\Gamma_\lambda,\chi)$-automorphic functions which are bounded on the fundamental domain $\mc F_\lambda$, and let $C^\infty_c(X_\lambda;V;\chi)$ be its subspace of functions which are compactly supported on $\mc F_\lambda$. We endow $C^\infty_c(X_\lambda;V;\chi)$ with the inner product
\begin{equation}\label{innerproduct}
 (f_1,f_2) \sceq \int_{\mc F_\lambda} \langle f_1(z), f_2(z)\rangle \dvol(z)\quad \big(f_1,f_2\in C^\infty_c(X_\lambda;V;\chi)\big),
\end{equation}
where $\langle\cdot,\cdot\rangle$ is the inner product on $V$, and $\dvol$ is the hyperbolic volume form. Note that $\chi$ being unitary yields that the definitions of $C^\infty(X_\lambda;V;\chi)$ and $C^\infty_c(X_\lambda;V;\chi)$ as well as the inner product \eqref{innerproduct} do not depend on the specific choice of the fundamental domain $\mc F_\lambda$ for $X_\lambda$. As usual, let 
\[
\mc H \sceq L^2(X_\lambda; V; \chi)
\]
denote the completion of $C^\infty_c(X_\lambda;V;\chi)$ with respect to the inner product $(\cdot,\cdot)$ defined in \eqref{innerproduct}. The Laplace-Beltrami operator
\[
 \Delta = - y^2 (\partial_x^2 + \partial_y^2)
\]
extends uniquely (Friedrich's extension) from
\[
 \{ f\in C^\infty(X_\lambda;V;\chi) \mid \text{ $f$ and $\Delta f$ are bounded on $\mc F_\lambda$} \}
\]
to a self-adjoint nonnegative definite operator on $\mc H$, which we also denote by $\Delta = \Delta(\Gamma_\lambda;\chi)$. If $f\in \mc H$ is an eigenfunction of $\Delta$, say $\Delta f = \mu f$, we branch its eigenvalue as $\mu=s(1-s)$ and call $s$ its \textit{spectral parameter}.

This branching is also useful when considering the resolvent of $\Delta$, hence the map
\[
R(s) \sceq R(s;\Gamma_\lambda;\chi) \sceq \big( \Delta - s(1-s) \big)^{-1}, \quad \Rea s > 1.
\]
Here, $s(1-s)$ is understood as the operator $s(1-s)\id_V$. This map, as a function of $s$, admits a meromorphic continuation to all of $\C$. Its poles are called \textit{resonances}.

For cofinite Hecke triangle groups $\Gamma_\lambda$, $\lambda \leq 2$, the \textit{cusp (vector) forms} in $\mc H$ are of particular importance, i.e., those $L^2$-eigenfunctions that decay rapidly towards any cusp of $X_\lambda$. More precisely, let 
\[
 V_\infty\sceq \{ v\in V \mid \chi(T_\lambda)v=v\}
\]
be the subspace of $V$ which consists of the vectors fixed by the stabilizer group $\Stab(\infty,\Gamma_\lambda)$ of the cusp $\infty$, and, for $\lambda=2$, let 
\[
 V_{-1}\sceq \{ v\in V \mid \chi(P_{-1})v=v\}
\]
be the subspace of $V$ which consists of the vectors fixed by the stabilizer group $\Stab(-1,\Gamma_\lambda)$ of the second cusp $-1$. Then an element $f\in \mc H$ is called a $(\Gamma_\lambda,\chi)$-\textit{cusp form} if $f$ is an eigenfunction of $\Delta$ and 
\[
 \int_0^\lambda \langle f(x+iy), v\rangle dx = 0
\]
for all $y>0$ and all $v\in V_\infty$, and, for $\lambda=2$, also
\[
 \int_0^{\frac{1}{2}} \left\langle f\left( \textbmat{1}{1}{-1}{0}.(x+iy)\right), v\right\rangle dx = 0
\]
for all $y>0$ and all $v\in V_{-1}$. We note that these conditions are void if $\chi$ is regular. We further note that for the trivial one-dimensional representation $(\C,\id)$, the $(\Gamma_\lambda,\id)$-cusp forms are the classical Maass cusp forms. 

\subsection{Selberg zeta functions}
Let $\Lambda(\Gamma_\lambda)$ denote the limit set of $\Gamma_\lambda$ and $\delta \sceq \dim\Lambda(\Gamma_\lambda)$ its Hausdorff dimension. The classical dynamical Selberg zeta function for $\Gamma_\lambda$ is defined by
\[
 Z(s) \sceq \prod_{\ell\in\Primlength} \prod_{k=0}^\infty \left( 1 - e^{-(s+k)\ell}\right),\quad \Rea  s > \delta,
\]
where $\Primlength$ denotes the primitive geodesic length spectrum of $X_\lambda$ with multiplicities. The Selberg zeta function for the automorphic Laplacian with nontrivial representation is best stated in its algebraic form, taking advantage of the bijection between periodic geodesics on $X_\lambda$ and hyperbolic conjugacy classes in $\Gamma_\lambda$. 

Let $g\in\Gamma_\lambda$. We denote its conjugacy class in $\Gamma_\lambda$ by $[g]$. Let $[\Gamma_\lambda]_h$ denote the set of all conjugacy classes of hyperbolic elements in $\Gamma_\lambda$. Further, if $g$ is not the identity element, then there is a maximal $n(g)\in\N$ such that $g= h^{n(g)}$ for some $h\in\Gamma_\lambda$.We call $g$ \textit{primitive} if $n(g) = 1$. We let $[\Gamma_\lambda]_p$ denote the set of all conjugacy classes of primitive hyperbolic elements. If $g$ is hyperbolic, then its \textit{norm} $N(g)$ is defined to be the square of its eigenvalue with larger absolute value. 

The Selberg zeta function for $(\Gamma_\lambda,\chi)$ is now defined by
\[
 Z(s,\chi) = \prod_{[g]\in [\Gamma_\lambda]_p} \prod_{k=0}^\infty \det\left( 1 - \chi(g) N(g)^{-(s+k)}\right), \quad\Rea s > \delta.
\]
Through analytic and number theoretic methods, it is known that $Z$ admits a meromorphic continuation to all of $\C$. We will provide an alternative proof in Section~\ref{fast} below. Moreover, it is known that the resonances and spectral parameters of $\Gamma_\lambda$ are contained in the zeros of $Z$.

\subsection{Special properties of Hecke triangle groups}

The Hecke triangle groups form a family that satisfies some special properties. They consist not only of a mixture of lattices and non-lattices which, in a certain sense, converge to each other, but also they mix arithmetic and non-arithmetic lattices.
Among the cofinite Hecke triangle groups $\Gamma_\lambda$ only those for $\lambda = 2\cos\frac{\pi}{q}$ with $q\in\{3,4,6\}$ and $q=\infty$, thus $\lambda = 2$, are arithmetic. 

The Phillips--Sarnak conjecture \cite{Phillips_Sarnak_cuspforms, Phillips_Sarnak_weyl} states that for generic lattices one expects the space of Maass cusp forms to be finite-dimensional. Hecke triangle groups clearly commute with the element $Q$ from \eqref{defQ} or, more obviously from Figure~\ref{funddoms}, with
\[
 J = \bmat{-1}{0}{0}{1}\colon z \mapsto -\overline z.
\]
This allows us to separate the odd and the even spectrum. Recall that a Maass cusp form $f$ is called \textit{odd} if $f(J.z) = -f(z)$, and \textit{even} if $f(J.z) = f(z)$. It is well-known that for any cofinite Hecke triangle groups, a Weyl law holds for the odd Maass cusp forms. In stark contrast, the results in \cite{Judge, Hillairet_Judge} in combination with the Phillips--Sarnak conjecture suggest  that even Maass cusp forms generically should not exist.

\section{Weighted discretizations and transfer operators}

\subsection{Cross sections}
We call a subset $\wh C$ of $SX_\lambda = \Gamma_\lambda\backslash S\h$ a \textit{cross section} for the geodesic flow on $X_\lambda$ if and only if the intersection between any geodesic and $\wh C$ is discrete in space and time, and each \textit{periodic} geodesic intersects $\wh C$ (infinitely often).

The Selberg zeta function only involves the length spectrum of the periodic geodesics, and heuristically, the Laplace eigenfunctions are determined by the (statistics of the) periodic geodesics only. Therefore, for our applications it is sufficient and even crucial to use this relaxed notion of cross section.

A \textit{set of representatives} for $\wh C$ is a subset $C'$ of $S\h$ such that the canonical quotient map $\pi\colon S\h \to SX_\lambda$ induces a bijection $C'\to \wh C$. For any $\wh v\in SX_\lambda$ let $\wh\gamma_{\wh v}$ denote the geodesic on $X_\lambda$ determined by 
\begin{equation}\label{determine}
 \wh\gamma'_{\wh v}(0) = \wh v.
\end{equation}
The \textit{first return map} of a cross section $\wh C$ is given by 
\[
 R\colon \wh C \to \wh C,\quad \wh v \mapsto \wh\gamma'_{\wh v}(t(\wh v)),
\]
whenever the \textit{first return time}
\[
 t(\wh v) \sceq \min\{ t>0 \mid \wh\gamma'_{\wh v}(t) \in \wh C\}
\]
exists. 

To characterize the set of unit tangent vectors for which the existence of the first return time is potentially problematic, let $\wh{\mc V}$ denote the set of geodesics on $X_\lambda$ which converge in the backward or forward time direction to a cusp or a funnel of $X_\lambda$. This means that the geodesic $\wh\gamma$ belongs to $\wh{\mc V}$ if and only if $\wh\gamma(\infty)$ or $\wh\gamma(-\infty)$ is ``in'' a cusp or a funnel. Let $\wh{\mc V}_f$ be the subset of geodesics on $X_\lambda$ which converge to a cusp or funnel in the forward time direction. Let $T^1\wh{\mc V}_f$ denote the set of unit tangent vectors which determine the geodesics in $\wh{\mc V}_f$, and let $T^1\mc V_f \sceq \pi^{-1}(T^1\wh{\mc V}_f)$. For $v\in S\h$ let $\gamma_v$ denote the geodesic on $\h$ determined by $\gamma'_v(0) = v$, and let $\base(v)\in\h$ denote the base point of $v$.

Let 
\[
 \bd \sceq \{ \gamma_v(\infty) \mid v\in T^1\mc V_f\} \quad \subseteq \R \cup \{\infty\}
\]
be the ``boundary part'' of the geodesic boundary of $\h$, that is the set of forward time endpoints of the geodesics determined by the elements in $T^1\mc V_f$. The set
\[
 \big(P^1(\R) \times \bd\big) \cup \big(\bd \times P^1(\R)\big)
\]
coincides with the set of endpoints $(\gamma(\infty),\gamma(-\infty))$ of the geodesics $\gamma$ in $\pi^{-1}(\wh{\mc V})$. Note that if the geodesic $\wh\gamma$ is contained in $\wh{\mc V}$ but not in $\wh{\mc V}_f$, then the reversely oriented geodesic is contained in $\wh{\mc V}_f$.

For any subset $I$ of $\R$ we set
\[
 I_\st \sceq I \setminus \bd.
\]

The cross sections $\wh C$ we constructed in \cite{Pohl_Symdyn2d} satisfy several special properties. They are not intersected at all by the geodesics in $\wh{\mc V}$ but by each other geodesic infinitely often both in backward and forward time direction. Therefore its first return map is defined everywhere. They have a set of representatives $C'$ which decomposes into finitely many sets $C'_\alpha$, $\alpha\in A$, each one consisting of a certain ``fractal-like'' set of unit tangent vectors whose base points form a vertical geodesic arc on $\h$ and all of who point into the same half space determined by this geodesic arc. More precisely, each $C'_\alpha$ is of the form
\[
 C'_\alpha = \{ v\in S\h \mid \Rea \base(v) = x_\alpha,\ \gamma_v(\infty) \in I_{\alpha,\st},\ \gamma_v(-\infty)\in \R_\st\setminus I_{\alpha,\st} \},
\]
where $I_{\alpha} = (x_\alpha,\infty)$ or $I_{\alpha} = (-\infty,x_\alpha)$ for some point $x_\alpha\in\R$ which corresponds to a cusp or funnel endpoint of $\Gamma_\lambda$.

Via the map $\tau\colon \wh C \to \R\times A$,
\[
 \tau(\wh v) \sceq (\gamma_v(\infty), \alpha) \qquad \text{for $v=\pi^{-1}(\wh v) \cap C'_\alpha$},
\]
the first return map $R\colon \wh C \to \wh C$ induces a discrete dynamical system $(D, F)$ on 
\[
 D = \bigcup_{\alpha\in A} I_{\alpha, \st}\times \{\alpha\}.
\]
The special structure of the sets $C'_\alpha$ implies that $F$ decomposes into finitely many \textit{submaps} (bijections) of the form 
\[
 \big( I_{\alpha,\st} \cap g^{-1}_{\alpha,\beta}. I_{\beta,\st}\big) \times \{\alpha\} \to I_{\beta,\st} \times \{\beta\},\quad (x,\alpha) \mapsto (g_{\alpha,\beta}.x, \beta),
\]
where $\alpha,\beta\in A$ and $g_{\alpha,\beta}$ is a specific element in $\Gamma_\lambda$. For any $v\in C'_\alpha$ there is a first future intersection between $\gamma_v(\R_{>0})$ and $\Gamma.C'$, say on $g_{\alpha,\beta}.C'_\beta$, which completely determines these submaps.

\subsection{Weights and weighted transfer operators}\label{sec:TO}

We use the representation $(V,\chi)$ of $\Gamma_\lambda$ to endow the discrete dynamical system $(D,F)$ with weights by adding to each submap a (local) weight:
\[
 \big( I_{\alpha,\st} \cap g^{-1}_{\alpha,\beta}. I_{\beta,\st}\big) \times \{\alpha\} \to I_{\beta,\st} \times \{\beta\},\quad (x,\alpha) \mapsto (g_{\alpha,\beta}.x, \beta), \quad \weight: \chi(g_{\alpha,\beta}^{-1})
\]
for all $\alpha,\beta\in A$. If $x\in ( I_{\alpha,\st} \cap g^{-1}_{\alpha,\beta}. I_{\beta,\st}) \times \{\alpha\}$, then we say that $x$ has the weight $\weight(x) = \chi(g_{\alpha,\beta}^{-1})$.

Given such a weighted discrete dynamical system $(D,F,\chi)$, the associated weighted transfer operator $\mc L_{s}$ with parameter $s\in\C$ is (formally) given by
\[
 \mc L_{s} f(x) \sceq \sum_{y\in F^{-1}(x)} \weight(y) |F'(y)|^{-s} f(y),
\]
acting on an appropriate space of functions $f\colon D\to V$ (to be adapted to the system and applications under consideration).

Due to the special form of our weighted systems, we can deduce a more explicit form for the transfer operators. To that end 
we set
\[
 j_s(g,x) \sceq \big( |\det g|\cdot (cx+d)^{-2} \big)^s
\]
for $s\in\C$, $g=\textbmat{a}{b}{c}{d}\in \PGL_2(\R)$ and $x\in\R$. Moreover, for a function $f\colon U\to V$ on some subset of $U$ of $P^1(\R)$, we define
\[
 \tau_s(g^{-1})f(x) \sceq \tau^V_s(g^{-1})f(x) \sceq j_s(g,x) f(g.x),
\]
whenever this makes sense. If $f$ is a function defined on sets of the form $(I_{\alpha,\st} \cap g^{-1}_{\alpha,\beta}. I_{\beta,\st}) \times \{\alpha\}$, then the action only takes place on $I_{\alpha,\st} \cap g^{-1}_{\alpha,\beta}. I_{\beta,\st}$, and the bit $\alpha$ changes according to its change in the submap. Further, we extend this definition to involve the representation by
\[
 \alpha_s(g) \sceq \chi(g^{-1})\tau_s^V(g),
\]
hence
\[
 \alpha_s(g)f(x) = j_s(g^{-1},x) \chi(g^{-1})\left(f(g^{-1}.x)\right).
\]
Then the (formal) transfer operator $\mc L_{s}$ becomes
\[
 \mc L_{s} f = \sum_{\alpha,\beta\in A} 1_{F(D_{\alpha,\beta})} \cdot \alpha_s(g_{\alpha,\beta})\big(f\cdot 1_{D_{\alpha,\beta}}\big),
\]
where $1_E$ denotes the characteristic function of the set $E$, and $D_{\alpha,\beta} \sceq (I_{\alpha,\st} \cap g^{-1}_{\alpha,\beta}. I_{\beta,\st}) \times \{\alpha\}$.

If $\Gamma_\lambda = \PSL_2(\Z)$ and $(V,\chi)$ is the representation that is induced from the trivial one-dimensional representation of some finite-index subgroup $\Lambda$ of $\PSL_2(\Z)$, then our weights reproduce the usage of $(V,\chi)$ in  \cite{Chang_Mayer_eigen, Chang_Mayer_extension,Deitmar_Hilgert}. There this specific representation was used to provide a clean bookkeeping tool of the cosets $\Lambda\backslash\PSL_2(\Z)$ to push the discretizations, transfer operators and period functions from $\PSL_2(\Z)$ to $\Lambda$. Our usage as weights directly in the submaps now allows us to accommodate more general representations.

\section{The slow systems and automorphic forms}\label{slow}

In the following we recall the cross sections from \cite{Pohl_Symdyn2d} and state the associated weighted discrete dynamical systems. Each of the submaps 
\[
 I_\st \times \{\alpha\} \to g.I_\st \times \{\beta\},\ (x,\alpha) \mapsto (g.x,\beta),\ \weight: \chi(g^{-1})
\]
of these systems can be continued to an analytic map on the ``analytic hull'' of $I_\st$, that is the minimal interval $I$ in $\R$ such that $I\setminus\bd$ coincides with $I_\st$.  The associated weighted transfer operators are here considered to act on functions defined on the analytic hulls.

If $\Gamma_\lambda$ is cofinite and $(V,\chi)$ is the trivial one-dimensional representation of $\Gamma_\lambda$, then we proved in \cite{Moeller_Pohl, Pohl_mcf_Gamma0p, Pohl_mcf_general} that the highly regular eigenfunctions with eigenvalue $1$ of the transfer operators are isomorphic to the Maass cusp forms. We expect this to hold for more general automorphic cusp forms. Moreover, for the non-cofinite Hecke triangle groups, we expect an intimate relation between the transfer operators and residues at resonances.

We divide our considerations into three classes. The first class consists of all cofinite Hecke triangle groups with parameter $\lambda < 2$. These are, in terms of complexity, the easiest ones. The cross sections for all these lattices consist of only one component, and they are almost identical. The only difference derives from the different sets of cuspidal points $\Gamma_\lambda.\infty$. The arising discrete dynamical systems are easy stated uniformly, even though they decompose into a different number of submaps, depending on the order of the elliptic point.

The second class consists of the Theta group $\Gamma_2$. The presence of two non-equivalent cusps and an elliptic point yields that the set of representatives for the cross section decomposes into three components. One immediately sees that only two components would be sufficient for a cross section. However, all three components are needed for a clean statement of the period functions  (i.e., the eigenfunctions of the transfer operator) for Maass cusp forms and in particular their regularity conditions.  

After surveying this result and its existing and expected generalizations to more general automorphic cusp forms, we investigate the reduced cross section for $\Gamma_2$, the relation between the arising transfer operators from the two cross sections, and provide the characterization of Maass cusp forms in this second system. More importantly, we will see that this reduced cross section is very similar to the cross sections for the non-cofinite Hecke triangle groups, our third class. As we will see, this similarity allows us to prove a convergence of transfer operators along a sequence of Hecke triangle groups.

Throughout we omit some dependencies from the notation, and we will use $(D,F)$ to denote any discrete dynamical system.

\subsection{Hecke triangle groups with parameter $\lambda<2$}\label{cross1}

Let $\Gamma = \Gamma_\lambda$ be a cofinite Hecke triangle group with parameter $\lambda=\lambda(q) = 2\cos\frac{\pi}{q}$, $q\in\N_{\geq 3}$. The set of representatives $C'$ in \cite{Pohl_Symdyn2d} for the cross section $\wh C$ of the geodesic flow on $\Gamma\backslash\h$ is 
\[
 C'\sceq \{ v\in S\h \mid \Rea\base(v) = 0,\ \gamma_v(-\infty) \in (-\infty,0)_\st,\ \gamma_v(\infty) \in (0,\infty)_\st \}.
\]
For $k\in\Z$ let
\[
 g_k \sceq  \big((TS)^k S\big)^{-1}.
\]
The induced discrete dynamical system $(D,F)$ is defined on $D=(0,\infty)_\st$ and decomposes, as it can be read off from Figure~\ref{Hecke_cof}, 
\begin{figure}[h]
\begin{center}
\includegraphics{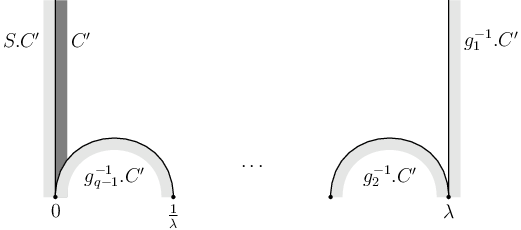}
\end{center}
\caption{Cross section and future intersections for $\Gamma_\lambda$, $\lambda<2$.}
\label{Hecke_cof}
\end{figure}
into the $q-1$ weighted submaps 
\[
 (g_k^{-1}.0, g_k^{-1}.\infty)_\st \to (0,\infty)_\st,\ \ x\mapsto g_k.x,\ \ \weight: \chi(g_k^{-1}),
\]
where $k=1,\ldots, q-1$.
The associated transfer operator $\mc L_{s,\chi}$ acts on $\Fct(\R_{>0};V)$ via
\[
 \mc L_{s,\chi} =  \sum_{k=1}^{q-1} \alpha_s(g_k).
\]

\begin{thm}[\cite{Moeller_Pohl}]\label{MP_slow}
Let $(V,\chi) = (\C,\id)$ be the trivial one-dimensional representation of $\Gamma_\lambda$ and let $s\in\C$, $\Rea s \in (0,1)$. Then the space of Maass cusp forms for $\Gamma_\lambda$ with spectral parameter $s$ is isomorphic to the space of eigenfunctions $f\in C^\omega(\R_{>0};\C)$ of $\mc L_{s,\chi}$ with eigenvalue $1$ for which the map
\begin{equation}\label{extends}
 \begin{cases}
  f & \text{on $\R_{>0}$}
\\
-\tau_s(S)f & \text{on $\R_{<0}$}
 \end{cases}
\end{equation}
extends smoothly to all of $\R$. If $u$ is a Maass cusp forms with spectral parameter $s$, then the associated eigenfunction $f\in C^\omega(\R_{>0};\C)$ is given by
\[
 f(t) = \int_0^{i\infty} [u, R(t,\cdot)^s],
\]
where the integration is along any path in $\h$ from $0$ to $i\infty$, and $R(t,z) \sceq \Ima\left(\frac{1}{t-z}\right)$, and
\[
 [u,v] = \frac{\partial u}{\partial z} \cdot v dz + u\cdot \frac{\partial v}{\partial\overline z} d\overline z. 
\]
\end{thm}

In number theoretical terms, the space of eigenfunctions of $\mc L_{s,\chi}$ in Theorem~\ref{MP_slow} constitutes period functions for the Maass cusp forms. Let $u$ be a Maass cusp form with corresponding period function $f$. The condition that $f$ is a $1$-eigenfunction of $\mc L_{s,\chi}$ corresponds to $u$ being an eigenfunction of the Laplace--Beltrami operator, and the regularity requirement $f\in C^\omega(\R_{>0};\C)$ corresponds to $u$ being real-analytic. To model the property that $u$ is rapidly decaying towards the cusp on the side of period functions, we first need to extend the period functions to almost all of $\R$ (actually of $P^1(\R)$). In Figure~\ref{Hecke_cof} we see that $S.C'$ is ``opposite'' to $C'$ in the sense that $S.C'\cup C'$ is disjoint and almost coincides with the complete unit tangent spaces at $i\R_{>0}$. Further, $S.C'$ is conversely oriented to $C'$. Therefore, $f$ is extended by $-\tau_s(S)f$ in \eqref{extends}. Now the cusp is represented by $0$, for which 
reason we need 
to request 
that \eqref{extends} is smooth at $0$. Since $\infty$ represents the same cusp, there is no additional requirement at $\infty$.

The proof of Theorem~\ref{MP_slow}, for which we refer to \cite{Moeller_Pohl} or \cite{Pohl_mcf_general}, makes crucial use of the characterization of Maass cusp forms in parabolic cohomology by Bruggeman, Lewis and Zagier \cite{BLZ_part2}. Such a characterization is not (yet) available for $(\Gamma_\lambda,\chi)$-automorphic cusp forms with nontrivial representation. However, for representations of $\PSL_2(\Z)$ that are induced from the trivial one-dimensional representation of a finite index subgroup, Deitmar and Hilgert \cite{Deitmar_Hilgert} proved an analogue of Theorem~\ref{MP_slow} using hyperfunction theory. We expect that results analogous to Theorem~\ref{MP_slow} are true in much more generality.

\begin{conj}
Let $(V,\chi)$ be any unitary finite-dimensional representation of $\Gamma_\lambda$, and $\Rea s \in (0,1)$. Then there is a bijection between the $(\Gamma_\lambda,\chi)$-automorphic cusp forms with spectral parameter $s$ and the $1$-eigenfunctions of the transfer operator $\mc L_{s,\chi}$ of the same regularity as in Theorem~\ref{MP_slow}.
\end{conj}

\subsection{The Theta group}\label{cross2}

Let $\Gamma\sceq \Gamma_2$ be the Theta group. 

\subsubsection{The original system}\label{crosstheta1}
The cross section $C'$ in \cite{Pohl_Symdyn2d} for $\Gamma$ is given by
\[
 C' \sceq C'_a \cup C'_b \cup C'_c,
\]
where
\begin{align*}
C'_a & \sceq \{ v\in S\h \mid \Rea\base(v) = -1,\ \gamma_v(-\infty)\in (-\infty,-1)_\st,\ \gamma_v(\infty)\in (-1,\infty)_\st\},
\\
C'_b & \sceq \{ v\in S\h \mid \Rea\base(v) = 1,\ \gamma_v(-\infty)\in (1,\infty)_\st,\ \gamma_v(\infty)\in (-\infty,1)_\st\}, \text{ and}
\\
C'_c & \sceq \{ v\in S\h \mid \Rea\base(v) = 0,\ \gamma_v(-\infty)\in (-\infty,0)_\st,\ \gamma_v(\infty)\in (0,\infty)_\st\}.
\end{align*}
Let 
\[
 k_1 \sceq T,\quad k_2 \sceq T^{-1}S = \bmat{2}{1}{-1}{0},\quad k_3\sceq TS = \bmat{2}{-1}{1}{0},\quad\text{and}\quad k_4\sceq S.
\]
\begin{figure}[h]
\begin{center}
\includegraphics{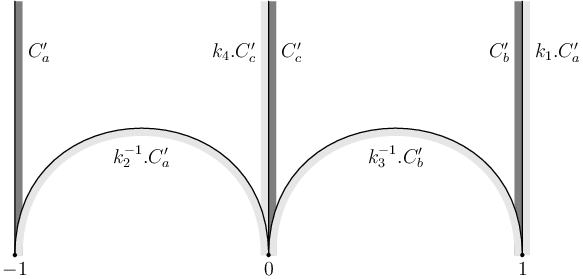}
\end{center}
\caption{Cross section and future intersections for $\Gamma_2$.}
\label{Theta1}
\end{figure}
From Figure~\ref{Theta1} it can be read off that the induced discrete dynamical system $(D,F)$ is defined on
\[
 D \sceq \big( (-1,\infty)_\st \times \{a\} \big) \cup \big( (-\infty,1)_\st \times \{b\} \big) \cup \big( (0,\infty)_\st \times \{c\} \big)
\]
and decomposes into the submaps
\begin{align*}
(-1,0)_\st \times \{a\} & \to (-1,\infty)_\st\times \{a\}, & (x,a) & \mapsto (k_2.x, a), && \weight: \chi(k_2^{-1}),
\\
(0,\infty)_\st \times \{a\} & \to (0,\infty)_\st\times \{c\}, & (x,a) & \mapsto (x,c), && \weight: \chi(1),
\\
(-\infty, 0)_\st\times \{b\} & \to (0,\infty)_\st\times \{c\}, & (x,b) & \mapsto (k_4.x, c), && \weight: \chi(k_4^{-1}),
\\
(0,1)_\st \times \{b\} & \to (-\infty, 1)_\st \times \{b\}, & (x,b) & \mapsto (k_3.x, b), && \weight: \chi(k_3^{-1}),
\\
(0,1)_\st \times \{c\} & \to (-\infty, 1)_\st \times \{b\}, & (x,c) & \mapsto (k_3.x, b), && \weight: \chi(k_3^{-1}),
\\
(1,\infty)_\st\times \{c\} & \to (-1,\infty)_\st\times \{a\}, & (x,c) & \mapsto (k_1^{-1}.x, a), && \weight: \chi(k_1^{-1}).
\end{align*}
For each $f\in \Fct(D;V)$ we let
\[
 f_a \sceq f\cdot 1_{(-1,\infty)_\st\times\{a\}},\quad f_b\sceq f\cdot 1_{(-\infty, 1)_\st\times\{b\}} \quad\text{and}\quad f_c\sceq f\cdot 1_{(0,\infty)_\st\times\{c\}}.
\]
We identify $f$ with the vector 
\[
 \begin{pmatrix} f_a \\ f_b \\ f_c \end{pmatrix},
\]
and then $(-1,\infty)_\st\times\{a\}$ with $(-1,\infty)_\st$, $(-\infty,1)_\st\times\{b\}$ with $(-\infty,1)_\st$, and $(0,\infty)_\st\times\{c\}$ with $(0,\infty)_\st$. The associated transfer operator $\mc L_{s,\chi}$ with parameter $s\in\C$ and weight $\chi$ is then represented by
\[
 \mc L_{s,\chi} = 
\begin{pmatrix}
\alpha_s(k_2) & 0 & \alpha_s(k_1^{-1})
\\
0 & \alpha_s(k_3) & \alpha_s(k_3)
\\
\alpha_s(1) & \alpha_s(k_4) & 0
\end{pmatrix}.
\]

The following theorem is a special case of the very general results in \cite{Pohl_mcf_general}.

\begin{thm}\label{Theta_slow}
Let $(V,\chi) = (\C,\id)$ be the trivial one-dimensional representation of $\Gamma$, and let $s\in\C$, $1>\Rea s > 0$. Then the Maass cusp forms $u$ for $\Gamma$ with spectral parameter $s$ are isomorphic to the function vectors $f = (f_a, f_b, f_c)^\top$ such that 
\[
f_a\in C^\omega\big((-1,\infty);\C\big), \quad f_b\in C^\omega\big( (-\infty,1); \C\big),\quad \text{and}\quad f_c\in C^\omega\big( (0,\infty);\C \big),
\]
and $f= \mc L_{s,\chi}f$, and the map
\[
 \begin{cases}
  f_c & \text{on $(0,\infty)$}
\\
-\tau_s(S)f_c & \text{on $(-\infty,0)$}
 \end{cases}
\]
extends smoothly to $\R$, and the map
\[
 \begin{cases}
  f_a & \text{on $(-1,\infty)$}
\\
-\tau_s(T^{-1})f_b & \text{on $(-\infty, -1)$}
 \end{cases}
\]
extends smoothly to $P^1(\R)$. The isomorphism from $u$ to $f$ is given by (see Theorem~\ref{MP_slow} for notation)
\begin{align*}
 f_a(t) & = \int_{-1}^{i\infty} [u, R(t,\cdot)^s],\quad f_b(t) = -\int_{1}^{i\infty} [u, R(t,\cdot)^s],
\intertext{and}
 f_c(t) & = \int_{0}^{i\infty} [u, R(t,\cdot)^s].
\end{align*}
\end{thm}

As in Section~\ref{cross1} we again expect that Theorem~\ref{Theta_slow} can be generalized to arbitrary $(\Gamma,\chi)$-automorphic cusp forms.

\subsubsection{The reduced system}\label{crosstheta2}

Figure~\ref{Theta1} shows that already $\pi(C'_a \cup C'_b)$ constitutes a cross section. The induced discrete dynamical system on
\[
 D \sceq \big( (-1,\infty)_\st\times\{a\} \big) \cup \big( (-\infty,1)_\st\times\{b\} \big)
\]
is given by the submaps
\begin{align*}
(-1,0)_\st\times \{a\} & \to (-1,\infty)_\st\times \{a\}, & (x,a) &\mapsto (k_2.x, a), && \weight: \chi(k_2^{-1})
\\
(0,1)_\st\times \{a\} & \to (-\infty, 1)_\st\times \{b\}, & (x,a) & \mapsto (k_3.x,b), && \weight: \chi(k_3^{-1})
\\
(1,\infty)_\st\times \{a\} & \to (-1,\infty)_\st\times \{a\}, & (x,a) & \mapsto (k_1^{-1}.x, a), && \weight: \chi(k_1)
\\
(-\infty, -1)_\st\times \{b\} & \to (-\infty, 1)_\st\times \{b\}, & (x,b) & \mapsto (k_1.x, b), && \weight: \chi(k_1^{-1})
\\
(-1,0)_\st\times \{b\} & \to (-1,\infty)_\st\times \{a\}, & (x,b) & \mapsto (k_2.x,a), && \weight: \chi(k_2^{-1})
\\
(0,1)_\st\times \{b\} & \to (-\infty,1)_\st\times\{b\}, & (x,b) & \mapsto (k_3.x,b), && \weight: \chi(k_3^{-1}).
\end{align*}
If we represent $f\in \Fct(D;V)$ by the function vector 
\[
 \begin{pmatrix} f_a \\ f_b \end{pmatrix}
\]
then the associated transfer operator is represented by
\[
 \mc L^{(r)}_{s,\chi} = 
\begin{pmatrix} 
\alpha_s(k_1^{-1}) + \alpha_s(k_2) & \alpha_s(k_2)
\\
\alpha_s(k_3) & \alpha_s(k_1) + \alpha_s(k_3) 
\end{pmatrix}.
\]

Even though the transfer operator $\mc L_{s,\chi}^{(r)}$ derives from a smaller cross section than $\mc L_{s,\chi}$, its highly regular $1$-eigenfunction vectors characterize the Maass cusp forms for $\Gamma$ in the case $(V,\chi) = (\C,\id)$ as we will show in the following.

\begin{lemma}\label{Theta_iso}
The map
\[
 \begin{pmatrix} f_a \\ f_b \end{pmatrix} \longleftrightarrow \begin{pmatrix} f_a \\ f_b \\ f_c \end{pmatrix}
\]
is an isomorphism between the $1$-eigenfunction vectors of $\mc L_{s,\chi}^{(r)}$ and the $1$-eigenfunction vectors of $\mc L_{s,\chi}$.
\end{lemma}

\begin{proof}
If $f=(f_a, f_b, f_c)^\top$ is a $1$-eigenfunction of $\mc L_{s,\chi}$, then 
\[
f_c = \alpha_s(1)f_a + \alpha_s(k_4)f_b.
\] 
\end{proof}

The combination of Theorem~\ref{Theta_slow} with Lemma~\ref{Theta_iso} yields the following Corollary.

\begin{cor}\label{Theta_rslow}
Let $(V,\chi) = (\C,\id)$ and $\Rea s \in (0,1)$. Then the Maass cusp forms for $\Gamma$ with spectral parameter $s$ are isomorphic to the $1$-eigenfunction vectors $f=(f_a,f_b)^\top$ of $\mc L_{s,\chi}^{(r)}$ which satisfy
\[
 f_a \in C^\omega\big( (-1,\infty); \C\big) \quad\text{and}\quad f_b \in C^\omega\big( (-\infty,1);\C\big),
\]
and for which the map
\[
 \begin{cases}
  f_a + \tau_s(S) f_b & \text{on $(0,\infty)$}
\\
 -\tau_s(S)f_a - f_b & \text{on $(-\infty,0)$}
 \end{cases}
\]
extends smoothly to $\R$, and the map
\[
 \begin{cases}
  f_a & \text{on $(-1,\infty)$}
\\
-\tau_s(T^{-1})f_b & \text{on $(-\infty,-1)$}
 \end{cases}
\]
extends smoothly to $P^1(\R)$. 
\end{cor}

The isomorphism in Corollary~\ref{Theta_rslow} is given as in Theorem~\ref{Theta_slow}.

\subsection{Non-cofinite Hecke triangle groups}\label{sec:slow_infinite}

Let $\Gamma$ be a non-cofinite Hecke triangle group, thus $\Gamma \sceq \Gamma_\lambda$ for some $\lambda>2$. The set of representatives $C'$ for the cross section $\wh C$ from \cite{Pohl_hecke_infinite} decomposes into two components
\[
 C' \sceq C'_a \cup C'_b
\]
which are
\begin{align*}
C'_a & \sceq \{ v\in S\h \mid \Rea\base(v)=-1,\ \gamma_v(-\infty)\in \R_\st,\ \gamma_v(\infty)\in (-1,\infty)_\st\}, \text{ and}
\\
C'_b & \sceq \{ v\in S\h \mid \Rea\base(v)=1,\ \gamma_v(-\infty)\in \R_\st,\ \gamma_v(\infty)\in (-\infty,1)_\st\}.
\end{align*}
The induced discrete dynamical system $(D,F)$ can be read off from Figure~\ref{Hecke_inf}, 
\begin{figure}[h]
\begin{center}
\includegraphics{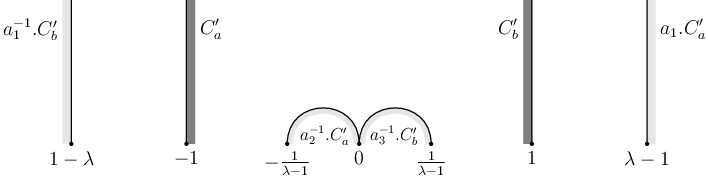}
\end{center}
\caption{Cross section and future intersections for $\Gamma_\lambda$, $\lambda>2$.}
\label{Hecke_inf}
\end{figure}
where  
\[
 a_1 \sceq T,\ a_2 \sceq T^{-1}S = \bmat{\lambda}{1}{-1}{0}\ \text{and}\ a_3\sceq TS = \bmat{\lambda}{-1}{1}{0}.
\]
Thus, it is defined on
\[
 D \sceq \big( (-1,\infty)_\st \times \{a\} \big) \cup \big( (-\infty, 1)_\st \times\{b\} \big)
\]
and given by the weighted submaps
\begin{align*}
(-1,0)_\st\times \{a\} & \to (-1,\infty)_\st\times\{a\}, &(x,a) & \mapsto (a_2.x,a), && \weight: \chi(a_2^{-1}),
\\
(0,1)_\st\times\{a\} & \to (-\infty,1)_\st\times\{b\}, &(x,a) & \mapsto (a_3.x,b), && \weight: \chi(a_3^{-1}),
\\
(-1+\lambda,\infty)_\st\times\{a\} & \to (-1,\infty)_\st\times\{a\}, &(x,a) & \mapsto (a_1^{-1}.x,a), && \weight: \chi(a_1),
\\
(-\infty, 1-\lambda)_\st\times\{b\} & \to (-\infty,1)_\st\times\{b\}, &(x,b) & \mapsto (a_1.x,b), && \weight: \chi(a_1^{-1}),
\\
(-1,0)_\st\times\{b\} & \to (-1,\infty)_\st\times\{a\}, &(x,b) & \mapsto (a_2.x,a), && \weight: \chi(a_2^{-1}),
\\
(0,1)_\st\times\{b\} & \to (-\infty,1)_\st\times\{b\}, &(x,b) & \mapsto (a_3.x,b), && \weight: \chi(a_3^{-1}).
\end{align*}
For $f\in \Fct(D;V)$ we set
\[
 f_1\sceq f\cdot 1_{(-1,\infty)_\st\times\{a\}} \quad\text{and}\quad f_2 \sceq f\cdot 1_{(-\infty,1)_\st\times\{b\}},
\]
and identify $f$ with
\[
\begin{pmatrix}
f_1 
\\
f_2  
\end{pmatrix}
\]
as well as $(-1,\infty)_\st\times\{a\}$ with $(-1,\infty)_\st$, and $(-\infty, 1)_\st\times\{b\}$ with $(-\infty,1)_\st$. The associated transfer operator $\mc L_{s,\chi}$ with parameter $s\in\C$ and weight $\chi$ is then represented by 
\[
 \mc L_{s,\chi} = 
\begin{pmatrix}
\alpha_s(a_2) + \alpha_s(a_1^{-1}) & \alpha_2(a_2) \\ \alpha_s(a_3) & \alpha_s(a_1) + \alpha_s(a_3)
\end{pmatrix}.
\]

In view of Theorems~\ref{MP_slow} and \ref{Theta_slow} we expect the following significance of the $1$-eigenfunction vectors of $\mc L_{s,\chi}$. This conjecture is also supported by Theorem~\ref{infinite_fast} below.

\begin{conj}
The (sufficiently regular) $1$-eigenfunction vectors of $\mc L_{s,\chi}$ determine the residues at the resonance $s$.
\end{conj}

\subsection{Convergence along sequences of Hecke triangle groups}\label{sequences}

The Phillips--Sarnak conjecture on the (non-)existence of Maass cusp forms is motivated by considerations using deformation theory along families of Fuchsian lattices. Also the results in \cite{Judge, Hillairet_Judge} are based on such a deformation theory. With this in mind, it is natural to ask how transfer operators behave along families of Fuchsian groups.

We recall from Section~\ref{sec:slow_infinite} that the elements in the discrete dynamical system for $\Gamma_\lambda$ with $\lambda>2$ are
\[
 a_1(\lambda) = \bmat{1}{\lambda}{0}{1},\quad a_2(\lambda) = \bmat{\lambda}{1}{-1}{0},\quad a_3(\lambda) = \bmat{\lambda}{-1}{1}{0}
\]
and the transfer operator is given by
\[
 \mc L_{s,\chi}^{(\lambda)} =
\begin{pmatrix}
\alpha_s(a_2(\lambda)) + \alpha_s(a_1(\lambda)^{-1}) & \alpha_s(a_2(\lambda))
\\
\alpha_s(a_3(\lambda)) & \alpha_s(a_1(\lambda)) + \alpha_s(a_3(\lambda))
\end{pmatrix}.
\]
In Section~\ref{crosstheta2} we see that for $\lambda=2$ we have $a_1(2)=k_1, a_2(2)=k_2, a_3(2) = k_3$ and that the transfer operator for the reduced system for the Theta group $\Gamma_2$ is just $\mc L_{s,\chi}^{(2)}$.

\begin{thm}\label{convergence}
We have the convergence
\[
 \mc L_{s,\chi}^{(\lambda)} \to \mc L_{s,\chi}^{(2)} \quad\text{as $\lambda \searrow 2$}
\]
in operator norm for any choice of norm on $\Fct((-1,\infty);V)\times \Fct((-\infty,1);V)$.
\end{thm}

We remark that Theorem~\ref{convergence} is very sensitive to the choice of the cross sections and their sets of representatives. Similar results hold for the transfer operators along the family of cofinite Hecke triangle groups if one uses a different cross section for the Theta group. The details will be discussed in the forthcoming Master thesis of A.\@ Adam (G\"ottingen).

 \section{The fast systems and the Selberg zeta functions}\label{fast}

In this section we discuss how to represent the Selberg zeta function $Z(s,\chi)$ for $\Gamma_\lambda$ (with the representation $(V,\chi)$) as a Fredholm determinant of a family of transfer operators arising from a discretization of the geodesic flow on $X_\lambda$. 

Each of the slow discrete dynamical systems in Section~\ref{slow} is expanding but none uniformly. This has the consequence that their associated transfer operators are not nuclear (on no nonzero Banach space) and hence do not admit Fredholm determinants. The reason for this non-uniformity in the expansion rate is the presence of submaps of the form
\[
  J_\alpha \times \{\alpha\} \to I_\alpha \times \{\alpha\},\quad (x,a) \mapsto (g.x,a),
\]
where $J_\alpha$ is a subset of $I_\alpha$ and $g$ is parabolic or elliptic. To overcome this problem we induce the slow discrete dynamical systems on these submaps. The outcome are discrete dynamical systems which branch into infinitely, but countably many submaps that are uniformly expanding. We call these systems \textit{fast}. 

The definition of transfer operator from Section~\ref{sec:TO} clearly carries over to fast discrete dynamical systems. But now the associated transfer operators involve infinitely many terms which makes necessary a discussion of their convergence and domains of definition. We show below that they constitute nuclear operators of order zero on direct sums of Banach spaces (with supremum norm) of the form
\begin{align*}
B(\mc C; V) &\sceq \{ \text{$f\colon \overline{\mc C} \to V$ continuous, $f\vert_{\mc C}$ holomorphic} \}, 
\end{align*}
where $\mc C$ is an open bounded disk in $P^1(\C)$. The action $\tau_s$ obviously extends to $B(\mc C; V)$. We continue to denote each arising discrete dynamical system by $(D,F)$.

\subsection{Hecke triangle groups with parameter $\lambda<2$}

Let $\Gamma \sceq \Gamma_\lambda$ be a cofinite Hecke triangle group with parameter $\lambda<2$. Recall the elements $g_k\in\Gamma$, $k=1,\ldots, q-1$, from Section~\ref{cross1}. The fast discrete dynamical system $(D,F)$ is defined on $D=(0,\infty)_\st$ and given by the weighted submaps
\begin{align*}
\left(g_k^{-1}.0, g_k^{-1}.\infty\right)_\st & \to (0,\infty)_\st, & x & \mapsto g_k.x, && \weight: \chi(g_k^{-1}),
\intertext{for $k=2,\ldots, q-2$, and, for $n\in\N$,}
\left(g_1^{-n}.0, g_1^{-(n+1)}.0\right)_\st & \to (0,\lambda)_\st, & x & \mapsto g_1^n.x, && \weight: \chi(g_1^{-n}),
\\
\left(g_{q-1}^{-(n+1)}.\infty, g_{q-1}^{-n}.\infty\right)_\st & \to \left(\frac1\lambda,\infty\right)_\st, & x & \mapsto g_{q-1}^n.x, && \weight: \chi(g_{q-1}^{-n}).
\end{align*}
We identify $f\in \Fct(D;\C)$ with the function vector $f=(f_1, f_r, f_{q-1})^\top$, where
\[
 f_1 \sceq f\cdot 1_{(\lambda,\infty)_\st},\quad f_r \sceq f\cdot 1_{(1/\lambda, \lambda)_\st}, \quad\text{and}\quad f_{q-1}\sceq f\cdot 1_{(0,1/\lambda)_\st}.
\]
Then the associated transfer operator with parameter $s\in\C$ is formally given by the matrix
\[
 \mc L_{s,\chi} =
\begin{pmatrix}
0 & \sum\limits_{k=2}^{q-2} \alpha_s(g_k) & \sum\limits_{n\in\N} \alpha_s(g_{q-1}^n)
\\
\sum\limits_{n\in\N} \alpha_s(g_1^n) &  \sum\limits_{k=2}^{q-2} \alpha_s(g_k) & \sum\limits_{n\in\N} \alpha_s(g_{q-1}^n)
\\
\sum\limits_{n\in\N} \alpha_s(g_1^n) &  \sum\limits_{k=2}^{q-2} \alpha_s(g_k) & 0
\end{pmatrix}.
\]

To find a domain of definition for $\mc L_{s,\chi}$ on which it becomes a nuclear operator, we need to work with neighborhoods of $\infty$ in $P^1(\C)$. To avoid dealing with changes of charts we conjugate the group $\Gamma$, the discretization, the discrete dynamical system and $\mc L_{s,\chi}$ with 
\[
 \mc T = \frac{1}{\sqrt{2}} \bmat{1}{-1}{1}{1}.
\]
Then $\Lambda \sceq \mc T \Gamma \mc T^{-1}$ is the conjugate lattice, the conjugate discrete dynamical system is $(E,G)$ defined on 
\[
 E\sceq (-1,1)_\st
\]
and given by the submaps from above, where each $g_k$ is changed to 
\[
 h_k \sceq \mc T g_k \mc T^{-1}
\]
and the weight is determined by the representation 
\[
 \eta\colon \Lambda \to U(V),\quad \eta(h) \sceq \chi(\mc T^{-1} h \mc T).
\]
Let 
\[
 \beta_s(h) \sceq \eta(h^{-1}) \tau_s(h)
\]
and
\begin{align*}
 E_1 &\sceq \mc T.(\lambda,\infty) = \left( \frac{\lambda -1}{\lambda+1}, 1\right),
\\
E_r & \sceq \mc T.\left(\frac{1}\lambda, \lambda\right) = \left( - \frac{\lambda-1}{\lambda+1}, \frac{\lambda-1}{\lambda+1}\right),
\\
E_{q-1} & \sceq \mc T.\left(0,\frac1\lambda\right) = \left(-1, -\frac{\lambda-1}{\lambda+1}\right).
\end{align*}
The conjugate transfer operator
 \[
 \wt{\mc L}_{s,\eta} =
\begin{pmatrix}
 0 & \sum\limits_{k=2}^{q-2} \beta_s(h_k) & \sum\limits_{n\in\N} \beta_s(h_{q-1}^n)
\\
\sum\limits_{n\in\N} \beta_s(h_1^n) & \sum\limits_{k=2}^{q-2} \beta_s(h_k) & \sum\limits_{n\in\N} \beta_s(h_{q-1}^n)
\\
\sum\limits_{n\in\N} \beta_s(h_1^n) & \sum\limits_{k=2}^{q-2} \beta_s(h_k) & 0
\end{pmatrix}
\]
is formally acting on the function vectors 
\[
f = 
\begin{pmatrix}
f_1 \colon E_1 \to V
\\
f_r \colon E_r \to V
\\
f_{q-1} \colon E_{q-1}\to V
\end{pmatrix}.
\]

As shown in \cite[Proposition~4.2]{Moeller_Pohl} there exist open bounded disks $\mc E_1, \mc E_r, \mc E_{q-1} \subseteq \C$ such that 

\begin{enumerate}[(i)]
\item $\overline{E_j} \subseteq \mc E_j$ for $j\in\{1,r,q-1\}$,
\item $J.\mc E_r = \mc E_r$ and $J.\mc E_1 = \mc E_{q-1}$, where $J = \mc T Q \mc T^{-1} = \textbmat{-1}{0}{0}{1}$,
\item for $k=2,\ldots, q-2$ we have $h_k^{-1}.\overline{\mc E}_1 \subseteq \mc E_r$, $h_k^{-1}.\overline{\mc E}_r \subseteq \mc E_r$ and $h_k^{-1}.\overline{\mc E}_{q-1} \subseteq \mc E_r$,
\item for $n\in\N$ we have $h_1^{-n}.\overline{\mc E}_r \subseteq \mc E_1$ and $h_1^{-n}.\overline{\mc E}_{q-1} \subseteq \mc E_1$,
\item for $n\in\N$ we have $h_{q-1}^{-n}.\overline{\mc E}_1 \subseteq \mc E_{q-1}$ and $h_{q-1}^{-n}.\overline{\mc E}_r\subseteq \mc E_{q-1}$,
\item for all $z\in\overline{\mc E}_1$ we have $\Rea z >-1$,
\item for all $z\in\overline{\mc E}_{q-1}$ we have $1>\Rea z$, and
\item for all $z\in\overline{\mc E}_r$ we have $1>\Rea z > -1$.
\end{enumerate}
Let 
\[
 B(\mc E;V) \sceq B(\mc E_1;V) \oplus B(\mc E_r;V) \oplus B(\mc E_{q-1};V).
\]

\begin{thm}\label{TO1}
\begin{enumerate}[{\rm (i)}]
\item\label{TOconv} For $\Rea s > \frac12$, the transfer operator $\wt{\mc L}_{s,\eta}$ is an operator on the Banach space $B(\mc E;V)$ and as such nuclear of order $0$.
\item\label{TOextension} The map $s \mapsto \wt{\mc L}_{s,\eta}$ extends meromorphically to all of $\C$. We also use $\wt{\mc L}_{s,\eta}$ to denote this extension. The poles are all simple and contained in $\frac{1}{2}(1-\N_0)$. For each pole $s_0$, there is a neighborhood $U$ of $s_0$ such that $\wt{\mc L}_{s,\eta}$ is of the form 
\[
 \wt{\mc L}_{s,\eta} = \frac{1}{s-s_0} \mc A_s + \mc B_s
\]
where the operators $\mc A_s$ and $\mc B_s$ are holomorphic on $U$, and $\mc A_s$ is of rank at most $4\sd(\eta)$.
\end{enumerate}
\end{thm}

\begin{proof}
The proof of \eqref{TOconv} is an easy adaption of the corresponding statement in \cite{Moeller_Pohl} for the trivial one-dimensional representation.  For \eqref{TOextension} we need to show that the maps 
\begin{align*}
\Psi_1 & \colon 
\left\{
\begin{array}{ccl}
\{\Rea s > \tfrac12\} & \to & \{\text{operators $B(\mc E_{q-1};V)\to B(\mc E_1;V)$}\}
\\
s & \mapsto & \sum\limits_{n\in\N} \beta_s(h_{q-1}^n)
\end{array}
\right.
\\
\Psi_2 & \colon\left\{
\begin{array}{ccl}
 \{\Rea s > \tfrac12\} & \to & \{\text{operators $B(\mc E_{q-1};V)\to B(\mc E_r;V)$}\}
\\
s & \mapsto & \sum\limits_{n\in\N} \beta_s(h_{q-1}^n)
\end{array}
\right.
\\
\Psi_3 & \colon\left\{
\begin{array}{ccl}
\{\Rea s > \tfrac12\} & \to & \{\text{operators $B(\mc E_{1};V)\to B(\mc E_r;V)$}\}
\\
s & \mapsto & \sum\limits_{n\in\N} \beta_s(h_1^n)
\end{array}
\right.
\\
\Psi_4 & \colon\left\{
\begin{array}{ccl}
 \{\Rea s > \tfrac12\} & \to & \{\text{operators $B(\mc E_{1};V)\to B(\mc E_{q-1};V)$}\}
\\
s & \mapsto & \sum\limits_{n\in\N} \beta_s(h_1^n)
\end{array}
\right.
\end{align*}
extend to meromorphic functions on $\C$ with values in nuclear operators of order $0$ and with poles as claimed. We provide these proofs for the latter two maps. The proofs for the former two maps are analogous. We start with a diagonalization. 

Since $\eta(h_1)$ is a unitary operator on $V$, there exists an orthonormal basis of $V$ with respect to which $\eta(h_1)$ is represented by a (unitary) diagonal matrix, say
\[
 \diag\left( e^{2\pi i a_1},\ldots, e^{2\pi i a_d} \right)
\]
with $a_1,\ldots, a_d\in\R$ and $d=\dim V$. The degree of singularity $\sd(\eta)$ is then the number of integral $a_j$ in this representation. We use the same basis of $V$ to represent any $f\in B(\mc E_k;V)$ ($k\in \{1,r,q-1\}$) as a vector of component functions 
\[
\begin{pmatrix} f_1 \\ \vdots \\ f_d\end{pmatrix} \colon \mc E_k \to \C^d. 
\]
In these coordinates, the operator on the right hand side in the definition of $\Psi_3$ and $\Psi_4$ becomes
\[
 \diag\left( \sum_{n\in\N} e^{2\pi i n a_1} \tau_s(h_1^n), \ldots, \sum_{n\in\N} e^{2\pi i n a_d} \tau_s(h_1^n)\right).
\]
Note that here $\tau_s = \tau_s^\C$.
Let
\[
 L_s \sceq \sum_{n\in\N} e^{2\pi i n a} \tau_s(h_1^n)
\]
with $a\in\R$. Then it suffices to show that for $\ell\in\{r, q-1\}$ and $a\in\R$, the maps
\[
 \left\{
\begin{array}{ccl}
\{\Rea s > \tfrac12\} & \to & \{\text{operators $B(\mc E_{1};\C)\to B(\mc E_\ell;\C)$}\}
\\
s & \mapsto & L_s
\end{array}
\right.
\]
admit entire extensions if $a\in \R\setminus\Z$, and, if $a\in\Z$, meromorphic extensions to all of $\C$ with simple poles all of which are contained in $\tfrac12(1-\N_0)$.

Note that $1 \in \mc E_1$ is the fixed point of the parabolic element $h_1^{-1}$. For $M\in \N_0$ let $P_M\colon B(\mc E_1;\C) \to B(\mc E_1;\C)$ be the operator which subtracts from a function its Taylor polynomial of degree $M$ centered at $1$, thus
\[
 P_M(g)(z) \sceq g_M(z) \sceq g(z) - \sum_{k=0}^M \frac{g^{(k)}(1)}{k!}(z-1)^k.
\]
We use the operator $P_M$ to write $L_s$ as a sum of two operators: 
\begin{equation}\label{L_split}
 L_s = L_s\circ (1-P_M) + L_s\circ P_M.
\end{equation}

We start by investigating the first term on the right hand side of \eqref{L_split}. For $n\in\N$ we have 
\[
 h_1^n = \frac12 \bmat{2+n\lambda}{-n\lambda}{n\lambda}{2-n\lambda}
\]
and hence
\[
 h_1^{-n}.z = 1 - \frac{2}{\lambda}\left( n + \frac{2}{\lambda(1-z)} \right)^{-1}.
\]
Let $\Rea s > \frac12$, $g\in B(\mc E_1;\C)$ and $z\in \mc E_r \cup \mc E_{q-1}$. Then 
\begin{align*}
\big(L_s \circ (1&-P_M)\big)g(z)  = L_s( g-g_M) (z)  = \sum_{k=0}^M \frac{g^{(k)}(1)}{k!}L_s\big( (z-1)^k \big)
\\
& = \sum_{k=0}^M (-1)^k \frac{g^{(k)}(1)}{k!} \frac{2^{2s+k}}{\lambda^{2s+k}(1-z)^{2s}}\sum_{n\in\N} e^{2\pi i n a} \left( n + \frac{2}{\lambda(1-z)}\right)^{-2s-k}.
\end{align*}
We recall the Lerch zeta function
\[
 \zeta(s,a,w) \sceq \sum_{n=0}^\infty \frac{e^{2\pi i n a}}{(n+w)^s},
\]
defined for $\Rea s > 1$, $a\in \R$ and $w\in \C\setminus (-\N_0)$. For $a\in\Z$, this is just the Hurwitz zeta function. A classical Riemann method shows that the Lerch zeta function extends meromorphically to all of $\C$ in the $s$-variable. For that one considers the contour integral
\[
 I(s,a,w) \sceq \frac{1}{2\pi i}\int_C \frac{z^{s-1} e^{wz}}{1-e^{z+2\pi i a}} dz,
\]
where $C$ is a path which begins at $-\infty$, goes once around the origin in positive direction, and returns to $-\infty$ such that it does not enclose any point in $2\pi i (-a + \Z)$ other than (possibly) $0$. Then $I(s,a,w)$ is entire in $s$, and 
\[
 \zeta(s,a,w) = \Gamma(s-1) I(s,a,w).
\]
Hence, for $a\in\Z$, the map $s\mapsto \zeta(s,a,w)$ extends to a meromorphic function on $\C$ with only a simple pole at $s=1$. For $a\notin\Z$, the extension of $s\mapsto \zeta(s,a,w)$ is entire.

The previous calculation now shows that
\[
 L_s(g-g_M)(z) = \sum_{k=0}^M (-1)^k \frac{g^{(k)}(1)}{k!} \frac{2^{2s+k}e^{2\pi i a}}{\lambda^{2s+k}(1-z)^{2s}}  \zeta\left( 2s+k, a, 1+\frac{2}{\lambda(1-z)}\right).
\]
From $1>\Rea z$ it follows that $\Rea \left(1+\frac{2}{\lambda(1-z)}\right)>0$. Thus, the properties of the Lerch zeta function imply that for $a\in\Z$ the map $L_s(g-g_M)(z)$ extends meromorphically in the $s$-variable to all of $\C$ with simple poles at $s=(1-k)/2$, $k=0,\ldots, M$. For $a\notin \Z$, the map $L_s(g-g_M)(z)$ extends to an entire function. In both cases, the extension of $L_s\circ (1-P_M)$ is nuclear of order $0$ as a finite rank operator.

For the investigation of the second term of \eqref{L_split} we fix $t>0$ such that the ball $B_t(1)$ in $\C$ with radius $t$ around $1$ is contained in $\mc E_1$. Let $g\in B(\mc E_1;\C)$ and $M\in\N$. There exists $C\in\R$ such that for all $z\in B_t(1)$, 
\[
 |g_M(z)| \leq C |z-1|^{M+1}.
\]
Further, since $h_1$ is parabolic with fixed point $1$,  we have
\[
 \lim_{n\to\infty} h_1^{-n}.z = 1
\]
for every $z\in\C$. Thus, for each $z\in \mc E_r\cup \mc E_{q-1}$ and for some $n_0 = n_0(z) \in\N$, we have
\begin{align*}
 \left|\big(L_s\circ P_M\big)g(z)\right| & = \left|\sum_{n\in\N} \frac{e^{2\pi i n a} 2^{2s}}{(n\lambda (1-z) + 2)^{2s}} g_M(h_1^{-n}.z)\right|
\\
& \leq \left| \sum_{n\leq n_0} \frac{e^{2\pi i n a} 2^{2s}}{(n\lambda (1-z) + 2)^{2s}} g_M(h_1^{-n}.z)\right|
\\
& \qquad + \sum_{n > n_0}\frac{2^{2\Rea s} e^{\pi |\Ima s|}}{|n\lambda(1-z) + 2|^{2\Rea s + M + 1}}.
\end{align*}
Hence, $(L_s\circ P_M)g(z)$ converges for $\Rea s > -\frac{M}{2}$. The operator $L_s\circ P_M$ is nuclear of order $0$ since $P_M$ is bounded.

Obviously, these arguments apply to each $M\in\N$, and hence $L_s$ extends meromorphically to $\C$ with simple poles at $s=(1-k)/2$, $k\in\N_0$, for $a\in \Z$, and it extends holomorphically to $\C$ for $a\notin\Z$. This proves \eqref{TOextension}.
\end{proof}

An immediate consequence of Theorem~\ref{TO1} is that the Fredholm determinant 
\[
\det(1-\wt{\mc L}_{s,\eta})
\]
exists and defines a holomorphic map on $\Rea s > \frac12$. For $\Rea s$ sufficiently large (here, $\Rea s > 1$), it is given by
\[
 \det(1-\wt{\mc L}_{s,\eta}) = \exp\Big( -\sum_{n\in\N} \frac1n \Tr \wt{\mc L}_{s,\eta} \Big).
\]
To show that it equals the Selberg zeta function $Z(s,\eta)$ we need the following lemma.

\begin{lemma}\label{traces}
Let $h\in\Lambda$ and let $\mc C$ be an open bounded subset of $\C$ such that $h.\overline{\mc C} \subseteq \mc C$ and such that the action of $h$ on $\mc C$ has a single fixed point. Consider the operator $\tau_s(h^{-1}) = \tau_s^V(h^{-1})$ as acting on the Banach space $B(\mc C; V)$. Then
\[
 \Tr(\beta_s(h^{-1})) = \frac{N(h)^{-s}}{1-N(h)^{-1}} \tr\eta(h).
\]
\end{lemma}

\begin{proof}
For $\tau_s^{\C}(h^{-1})\colon B(\mc C;\C) \to B(\mc C;\C)$ we have (see \cite{Ruelle_zeta})
\[
 \Tr \tau_s^{\C}(h^{-1}) = \frac{N(h)^{-s}}{1-N(h)^{-1}}.
\]
Since $\eta(h)$ is a unitary operator on $V$, we find a basis of $V$ with respect to which $\eta(h)$ is represented by a diagonal matrix, say by
\[
 \diag(a_1,\ldots, a_d)
\]
with $a_1,\ldots, a_d\in \C$, $d=\dim V$. We use the same basis of $V$ to define an isomorphism (of Banach spaces)
\[
 B(\mc C; V) \cong \bigoplus_{j=1}^d B(\mc C; \C).
\]
Under this isomorphism, the operator $\tau_s^{V}(h)$ acts diagonally, i.e.\@
\[
 \tau_s^V(h^{-1}) \cong \bigoplus_{j=1}^d \tau_s^{\C}(h^{-1}).
\]
Hence
\[
 \eta(h)\tau_s^V(h^{-1}) \cong \bigoplus_{j=1}^d a_j \tau_s^{\C}(h^{-1})
\]
and
\begin{align*}
 \Tr \eta(h)\tau_s^V(h^{-1}) & = \sum_{j=1}^d a_j \Tr\tau_s^{\C}(h^{-1}) = \frac{N(h)^{-s}}{1- N(h)^{-1}} \sum_{j=1}^d a_j = \frac{N(h)^{-s}}{1- N(h)^{-1}} \tr \eta(h).
\end{align*}
\end{proof}

\begin{thm}\label{ZetaFredholm}
\begin{enumerate}[{\rm (i)}]
\item\label{zeta1} For $\Rea s > 1$ we have $Z(s, \eta) = \Det(1 - \wt{\mc L}_{s,\eta})$.
\item\label{FDextension} The Fredholm determinant $\Det( 1 - \wt{\mc L}_{s,\eta})$ extends to a meromorphic function on $\C$ whose poles are contained in $\frac12(1-\N_0)$. The order of a pole is at most $4\sd(\eta)$.
\end{enumerate}
\end{thm}

\begin{proof}
We start with the proof of \eqref{zeta1}. Let $n\in\N$ and suppose that $h=s_0\ldots s_{n-1}$ is a word over the alphabet
\begin{equation}\label{alphabet}
 \{ h_1^m, h_2,\ldots, h_{q-2}, h_{q-1}^m \mid m\in\N\}.
\end{equation}
We call $h$ \textit{reduced} if it does not contain a subword of the form $h_1^{m_1}h_1^{m_2}$ or $h_{q-1}^{m_1}h_{q-1}^{m_2}$, $m_1,m_2\in\N$. We say that $h$ is \textit{regular} if both $h$ and $hh$ are reduced. The length of $h$ is $n$. 

Since the semigroup $H$ in $\Lambda$ generated by \eqref{alphabet} is free (see \cite{Moeller_Pohl}), each element $h\in H$ corresponds to at most one word in this alphabet. For this reason, we will identify the elements in $H$ with their corresponding words.

Let $P_n$ denote the set of regular words of length $n$. An immediate consequence of \cite[Proposition~4.1]{Moeller_Pohl} is the identity
\[
 \Tr \wt{\mc L}_{s,\eta} = \sum_{h\in P_n} \Tr\beta_s(h).
\]
Further, let $g\in\Lambda$ be hyperbolic. By \cite{Moeller_Pohl}, the conjugacy class $[g]$ contains at least one regular representative $h\in H$. Let $h_0$ be a primitive hyperbolic element such that $h_0^n = h$ for some $n\in\N$. Then $h_0\in H$. The word length $w(h)$ of $h$ and the word length $w(h_0)$ of $h_0$ are independent of the chosen representatives. We define $w(g)\sceq w(h)$ to be the \textit{word length} of $g$, and $p(g)\sceq w(h_0)$ the  \textit{primitive word length} of $g$. Then there are exactly $p(g)$ representatives of $[g]$ in $P_{w(g)}$. Note that $n(g) = w(g)/p(g)$.

For $\Rea s > 1$ we now have
\begin{align*}
\ln Z(s) & = \sum_{[g]\in [\Gamma]_p} \sum_{k=0}^\infty \ln \left( \det\left(1-\eta(g)N(g)^{-(s+k)}\right)\right)
\\
& = -\sum_{[g]\in [\Gamma]_p} \sum_{\ell=1}^\infty \frac{1}{\ell} \frac{N(g^\ell)^{-s}}{1-N(g^\ell)^{-1}} \tr\left( \eta(g^\ell) \right)
\\
& = -\sum_{[g]\in [\Gamma]_h} \frac{1}{n(g)} \frac{N(g)^{-s}}{1-N(g)^{-1}} \tr\eta(g)
\\
& = -\sum_{[g]\in [\Gamma]_h} \frac{p(g)}{w(g)} \frac{N(g)^{-s}}{1-N(g)^{-1}} \tr\eta(g)
\\
& = -\sum_{w=1}^\infty \frac1w \sum_{h^{-1}\in P_w} \frac{N(h)^{-s}}{1-N(h)^{-1}} \tr\eta(h)
\\
& = -\sum_{w=1}^\infty \frac1w \sum_{h\in P_w} \Tr\beta_s(h)
\\
& = -\sum_{w=1}^\infty \frac1w \Tr\wt{\mc L}_{s,\eta}^w
\\
& = \ln \left( \det(1-\wt{\mc L}_{s,\eta})\right).
\end{align*}
This proves \eqref{zeta1}. Statement \eqref{FDextension} follows now immediately from Theorem~\ref{TO1}\eqref{TOextension}.
\end{proof}

\begin{cor}
For regular $\eta$, the map $s \mapsto \wt{\mc L}_{s,\eta}$ and the Fredholm determinant $\Det( 1 - \wt{\mc L}_{s,\eta})$ extend to entire functions.
\end{cor}

For $\PSL_2(\Z)$ it is known (see \cite{Chang_Mayer_transop, Lewis_Zagier}) that the $1$-eigenfunctions of the ``slow'' transfer operators from Section~\ref{cross1} and the ``fast'' transfer operators developed in this section are isomorphic. Their proof takes advantage of special properties of the modular group. However, by geometric consideration we expect that the same relation holds for general Hecke triangle groups.

\subsection{The Theta group}

Let $\Gamma \sceq \Gamma_2$ be the Theta group. Recall the elements $k_1, k_2, k_3\in \Gamma$ from Section~\ref{cross2}. The first step from the original slow system for $\Gamma$ in Section~\ref{crosstheta1} towards a fast system has already been made in Section~\ref{crosstheta2}, where we eliminated the acting elliptic element $S$ and the action by $1\in\Gamma$. The additional induction on the parabolic elements gives rise to the fast discrete dynamical system $(D,F)$ which is defined on
\begin{align*}
D &\sceq \big( (-1,\infty)_\st\times \{a\}\big) \cup \big( (-\infty,1)_\st\times\{b\}\big)
\\
&\ =
\big( (0,1)_\st\times\{a\} \big) \cup \big( (-1,0)_\st\times \{a\}\big)
 \cup \big( (1,\infty)_\st\times\{a\}\big) 
\\
& \qquad  \cup \big( (-1,0)_\st\times\{b\} \big)  \cup \big( (-\infty,-1)_\st \times \{b\} \big) \cup \big( (0,1)_\st\times \{b\} \big)
\end{align*}
and determined by the weighted submaps 
\begin{align*}
(0,1)_\st\times \{a\} & \to (-\infty,1)_\st\times\{b\}, & (x,a) & \mapsto (k_3.x, b), && \weight: \chi(k_3^{-1})
\\
(-1,0)_\st\times \{b\} & \to (-1,\infty)_\st\times \{a\}, & (x,b) & \mapsto (k_2.x,a), && \weight: \chi(k_2^{-1})
\intertext{and for $n\in\N$}
\left(k_2^{-(n+1)}.\infty, k_2^{-n}.\infty\right)_\st\times \{a\} & \to (0,\infty)_\st\times \{a\}, & (x,a) & \mapsto (k_2^n.x, a), && \weight: \chi(k_2^{-n})
\\
\left(k_1^n.(-1), k_1^{n+1}.(-1)\right)_\st\times\{a\} & \to (-1,1)_\st\times \{a\}, & (x,a) & \mapsto (k_1^{-n}.x,a), && \weight: \chi(k_1^{n})
\\
\left(k_1^{-(n+1)}.1, k_1^{-n}.1\right)_\st\times\{b\} & \to (-1,1)_\st\times \{b\}, & (x,b) & \mapsto (k_1^n.x, b), && \weight: \chi(k_1^{-n})
\\
\left(k_3^{-n}.\infty, k_3^{-(n+1)}.\infty\right)_\st\times \{b\} & \to (-\infty,0)_\st\times\{b\}, & (x,b) & \mapsto (k_3^n.x, b), && \weight: \chi(k_3^{-n}).
\end{align*}
Note that 
\[
 \bigcup_{n\in\N} \left(k_2^{-(n+1)}.\infty, k_2^{-n}.\infty\right)_\st\times \{a\} = (-1,0)_\st\times \{a\},
\]
and analogously for the other infinite branches. As before, we identify $f\in \Fct(D;V)$ with the function vector $(f_1, \ldots, f_6)^\top$, where
\begin{align*}
f_1 & \sceq f \cdot 1_{(-1,0)_\st\times\{a\}}
\\
f_2 & \sceq f \cdot 1_{(0,1)_\st\times\{a\}}
\\
f_3 & \sceq f \cdot 1_{(1,\infty)_\st\times\{a\}}
\\
f_4 & \sceq f \cdot 1_{(-\infty,-1)_\st\times\{b\}}
\\
f_5 & \sceq f \cdot 1_{(-1,0)_\st\times\{b\}}
\\
f_6 & \sceq f \cdot 1_{(0,1)_\st\times\{b\}},
\end{align*}
and we may omit the $\times\{a\}$ and $\times\{b\}$. Then the associated transfer operator is (formally) represented by the matrix 
\[
 \mc L_{s,\chi} = 
\begin{pmatrix}
& & \sum\limits_{n\in\N} \alpha_s(k_1^{-n}) & & \alpha_s(k_2) & 
\\
\sum\limits_{n\in\N} \alpha_s(k_2^n) & & \sum\limits_{n\in\N}\alpha_s(k_1^{-n}) & & \alpha_s(k_2) & 
\\
\sum\limits_{n\in\N} \alpha_s(k_2^n) & & & & \alpha_s(k_2) &
\\
& \alpha_s(k_3) & & & & \sum\limits_{n\in\N} \alpha_s(k_3^n)
\\
& \alpha_s(k_3) & & \sum\limits_{n\in\N} \alpha_s(k_1^n) & & \sum\limits_{n\in\N} \alpha_s(k_3^n)
\\
& \alpha_s(k_3) & & \sum\limits_{n\in\N} \alpha_s(k_1^n) & &  
\end{pmatrix}.
\]

For $a,b\in P^1_\R$, $a\not=b$, let $B(a,b)$ denote the open ball in $P^1(\C)$ which passes through $a$ and $b$. If $a>b$, then $B(a,b)$ is understood to contain $\infty$.

Let 
\begin{align*}
 \mc D_1 &\sceq B\left( -\frac32, \frac12 \right),\quad \mc D_2 \sceq B\left( -\frac12, \frac52\right), \quad  \mc D_3 \sceq B(0,-10),
\\
\mc D_4 &\sceq B(10,0),\quad  \mc D_5 \sceq B\left(-\frac52, \frac12\right),\quad \mc D_6 \sceq B\left(-\frac12, \frac32\right).
\end{align*}
Let 
\[
 B(\mc D;V) \sceq \bigoplus_{j=1}^6 B(\mc D_j;V).
\]

Analogously to Theorems~\ref{TO1} and \ref{ZetaFredholm} one shows the following relations between the transfer operator family and the Selberg zeta functions. Note that the order of poles differs from that in Theorems~\ref{TO1} and \ref{ZetaFredholm}.

\begin{thm}
\begin{enumerate}[{\rm (i)}]
\item For $\Rea s > \frac12$, the transfer operator $\mc L_{s,\chi}$ is an operator on $B(\mc D;V)$ and as such nuclear of order $0$.
\item For $\Rea s > 1$, we have $Z(s,\chi) = \det(1-\mc L_{s,\chi})$.
\item The map $s\mapsto \mc L_{s,\chi}$ extends meromorphically to all of $\C$. The poles are all simple and contained in $\frac12(1-\N_0)$. For each pole $s_0$, there is a neighborhood $U$ of $s_0$ such that the meromorphic extension is of the form 
\[
 \mc L_{s,\chi} = \frac{1}{s-s_0} \mc A_s + \mc B_s,
\]
where the operators $\mc A_s$ and $\mc B_s$ are holomorphic on $U$, and $\mc A_s$ is of rank at most $6\sd(\chi)$.
\item The Fredholm determinant $\det(1-\mc L_{s,\chi})$ extends to a meromorphic function on $\C$ whose poles are contained in $\frac12(1-\N_0)$. The order of a pole is at most $6\sd(\chi)$.
\end{enumerate}
\end{thm}

\subsection{Non-cofinite Hecke triangle groups}

Let $\Gamma\sceq \Gamma_\lambda$ be a non-cofinite Hecke triangle group, thus $\lambda > 2$. Recall the elements $a_1, a_2, a_3\in \Gamma$ from Section~\ref{sec:slow_infinite}. The fast discrete dynamical system $(D,F)$ for $\Gamma$ is defined on 
\[
 D \sceq \big( (-1,\infty)_\st\times\{a\} \big) \cup \big( (-\infty,1)_\st\times \{b\} \big)
\]
and given by the weighted submaps
\begin{align*}
(-1,0)_\st\times \{a\} & \to (-1,\infty)_\st\times\{a\}, & (x,a) & \mapsto (a_2.x,a), && \weight: \chi(a_2^{-1}),
\\
(0,1)_\st\times\{a\} & \to (-\infty,1)_\st\times\{b\}, & (x,a) & \mapsto (a_3.x,b), && \weight: \chi(a_3^{-1}),
\\
(-1,0)_\st\times \{b\} & \to (-1,\infty)_\st\times \{a\}, & (x,b) & \mapsto (a_2.x,a), && \weight: \chi(a_2^{-1}),
\\
(0,1)_\st\times \{b\} & \to (-\infty,1)_\st\times \{b\}, & (x,b) & \mapsto (a_3.x,b), && \weight: \chi(a_3^{-1}),
\end{align*}
and, for $n\in\N$,
\begin{align*}
(-1+n\lambda, -1 + (n+1)\lambda)_\st\times \{a\} & \to (-1, -1+\lambda)_\st\times\{a\},
\\ (x,a) & \mapsto (a_1^{-n}.x,a), \qquad \weight: \chi(a_1^{n}),
\\
(1-(n+1)\lambda, 1-n\lambda)_\st\times\{b\} & \to (1-\lambda,1)_\st\times\{b\}, 
\\
(x,b) & \mapsto (a_1^n.x,b), 
\qquad \weight: \chi(a_1^{-n}).
\end{align*}
The associated transfer operator with parameter $s\in\C$ is represented by the matrix
\[
\mc L_{s,\chi} = 
\begin{pmatrix}
\alpha_s(a_2) & \sum\limits_{n\in\N} \alpha_s(a_1^{-n}) & & \alpha_s(a_2)
\\
\alpha_s(a_2) &  & & \alpha_s(a_2)
\\
\alpha_s(a_3) &  & & \alpha_s(a_3)
\\
\alpha_s(a_3) &  & \sum\limits_{n\in\N} \alpha_s(a_1^{n})  & \alpha_s(a_3)
\end{pmatrix}
\]
acting formally on the function vectors
\[
 f = 
\begin{pmatrix}
f_1\colon (-1,1)_\st\to V
\\
f_2 \colon (-1+\lambda,\infty)_\st\to V
\\
f_3\colon (-\infty, 1-\lambda)_\st\to V
\\
f_4\colon (-1,1)_\st\to V
\end{pmatrix}.
\]
Let 
\[
 \mc D_1 \sceq B(-1,1), \quad \mc D_2 \sceq B\left( \frac{5\lambda-4}{6}, -\frac{\lambda}{2}\right), \quad\text{and}\quad \mc D_3 \sceq B\left( \frac{\lambda}{2}, \frac{4-5\lambda}{6}\right),
\]
and set
\[
 B(\mc D;V) \sceq B(\mc D_1;V) \oplus B(\mc D_2;V) \oplus B(\mc D_3;V) \oplus B(\mc D_1;V).
\]

Analogously to Theorems~\ref{TO1} and \ref{ZetaFredholm} the following results are shown.

\begin{thm}\label{infinite_fast}
\begin{enumerate}[{\rm (i)}]
\item For $\Rea s > \frac12$, the transfer operator $\mc L_{s,\chi}$ defines a nuclear operator on $B(\mc D;V)$ of order $0$. The map $s\mapsto \mc L_{s,\chi}$ extends meromorphically to all of $\C$ with possible poles located at $s=(1-k)/2$, $k\in\N_0$. All poles are simple. For each pole $s_0$, there is a neighborhood $U$ of $s_0$ such that $\mc L_{s,\chi}$ is of the form
\[
 \mc L_{s,\chi} = \frac{1}{s-s_0}\mc A_s + \mc B_s,
\]
where the operators $\mc A_s$ and $\mc B_s$ are holomorphic on $U$, and $\mc A_s$ is of rank at most $2\sd(\chi)$.
\item For $\Rea s > \max( \dim\Lambda(\Gamma), \frac12)$ we have $Z(s,\chi) = \det(1-\mc L_{s,\chi})$. Moreover, the Fredholm determinant $s\mapsto \det(1-\mc L_{s,\chi})$ extends to a meromorphic function on $\C$ with possible poles located in $\frac12(1-\N_0)$. The order of a pole is at most $2\sd(\chi)$.
\end{enumerate}

\end{thm}

\section{Symmetry reduction}\label{symmetryreduction}

Suppose that the representation $(V,\chi)$ of $\Gamma_\lambda$ decomposes into a direct (finite) sum 
\begin{equation}\label{decomp}
 \chi = \bigoplus_{j=1}^m \chi_j
\end{equation}
of representations $(V_j,\chi_j)$ of $\Gamma_\lambda$. Then the Selberg zeta function factors accordingly \cite[Theorem~7.2]{Venkov_book2}:
\begin{equation}\label{selbergfactor}
  Z(s,\chi) = \prod_{j=1}^m Z(s,\chi_j).
\end{equation}
The structure of the interaction of the representation $\chi$ with the action $\tau_s$ yields that this kind of factorization already happens at the level of the transfer operators, and that \eqref{selbergfactor} is an immediate consequence of this more general result.

\begin{thm}
Suppose that $\chi$ decomposes as in \eqref{decomp}, and let $\mc L_{s,\chi}$ be a transfer operator of the form as in Section~\ref{slow} or \ref{fast}. Then 
\begin{equation}\label{splitTO}
 \mc L_{s,\chi} = \bigoplus_{j=1}^m \mc L_{s,\chi_j}.
\end{equation}
Moreover, 
\begin{equation}\label{splitZeta}
 Z(s,\chi) = \prod_{j=1}^m Z(s,\chi_j).
\end{equation}
\end{thm}

\begin{proof}
Obviously, we have 
\[
 \Fct(D;V) = \bigoplus_{j=1}^m \Fct(D;V_j) \quad\text{and}\quad B(\mc D; V) = \bigoplus_{j=1}^m B(\mc D; V_j).
\]
For any $g\in \Gamma_\lambda$ and any function $f$, the map $\tau_s(g)$ only acts on the argument of $f$, while $\chi(g)$ only acts on the vector $f(z)$. Therefore, 
\[
 \alpha_s^V(g) = \chi(g)\tau_s^V(g) = \bigoplus_{j=1}^m \chi_j(g) \tau_s^{V_j}(g) = \bigoplus_{j=1}^m \alpha_s^{V_j}(g).
\]
This implies \eqref{splitTO}. Then \eqref{splitZeta} follows immediately from
\begin{align*}
 Z(s,\chi) &= \det(1 - \mc L_{s,\chi}) 
\\
& = \det\big( 1 - \bigoplus_{j=1}^m \mc L_{s,\chi_j} \big)= \prod_{j=1}^m \det(1- \mc L_{s,\chi_j})
\\
& = \prod_{j=1}^m Z(s,\chi_j).
\end{align*}
\end{proof}

\section{Billiard flows}\label{billiard}

Recall the elements 
\[
 Q = \bmat{0}{1}{1}{0} \quad\text{and}\quad J = \bmat{-1}{0}{0}{1}
\]
in $\PGL_2(\R)$. The discrete dynamical systems and the transfer operators from Sections~\ref{slow} and \ref{fast} commute with the action of $Q$ and $\alpha_s(Q)$ for $\Gamma_\lambda$ with $\lambda<2$ respectively with the action of $J$ and $\alpha_s(J)$ for $\Gamma_\lambda$ with $\lambda\geq 2$. As shown in \cite{Moeller_Pohl} for $\Gamma_\lambda$ with $\lambda<2$ with the trivial representation $(\C,\id)$ this exterior symmetry allows us to characterize even/odd Maass cusp forms as those period functions that are invariant/anti-invariant under $\tau_s(Q)$, and it results in a factorization of the Selberg zeta function.

An approach from a higher point of view was taken on in \cite{Pohl_hecke_spectral} and \cite{Pohl_hecke_infinite}, where $\Gamma_\lambda$ was extended by these symmetries and then the billiard flow was considered. In the following we briefly survey this approach and extend it to include non-trivial representations.

Let 
\[
 \wt\Gamma_\lambda \sceq \left\langle Q, \Gamma_\lambda\right\rangle = \left\langle J, \Gamma_\lambda\right\rangle \ \leq \PGL_2(\R)
\]
denote the triangle group underlying the Hecke triangle group. Further, we extend the representation $(V,\chi)$ of $\Gamma_\lambda$ to a representation of $\wt\Gamma_\lambda$, also denoted by $\chi$, by defining $\chi(Q)$ (or, equivalently, $\chi(J)$).

We consider the billiard flow on $\wt\Gamma_\lambda\backslash\h$, that is, the geodesic flow on $\wt\Gamma_\lambda\backslash\h$ considered as an orbifold. We inherit the cross sections for the geodesic flow on $\Gamma_\lambda\backslash\h$ to $\wt\Gamma_\lambda\backslash\h$ by taking ``half'' of it and then use the arising discrete dynamical systems to define the transfer operator families. For example, the cross section for $\Gamma_\lambda$, $\lambda\geq 2$, is only $C'_a$ from Section~\ref{cross2} respectively \ref{sec:slow_infinite}. For $\Gamma_\lambda$ with $\lambda= 2\cos \pi/q$ and $q$ odd, it is
\begin{equation}\label{billcross}
 C' \sceq \{ v\in S\h \mid \Rea\base(v) = 0,\ \gamma_v(\infty)\in (0,1]_\st,\ \gamma_v(-\infty)\in (-\infty,0)_\st\}.
\end{equation}
For $\Gamma_\lambda$ with $\lambda= 2\cos \pi/q$ and $q$ even, the point $1\in\R$ is an endpoint of periodic geodesics. For that reason one cannot simply use \eqref{billcross} as cross section but rather needs an average over all reasonable choices for $C'$ to achieve that the tangent vectors belonging to these special periodic geodesics are ``inner vectors''. We refer to \cite{Pohl_hecke_spectral} for details. The upshot is that the contribution from the element $g_{q/2}$ gets an additional weight of $1/2$. We also refer to \cite{Pohl_hecke_spectral, Pohl_hecke_infinite} for the explicit formulas for the arising discrete dynamical systems which only need to be extended by the obvious weights. The discrete dynamical systems for $\wt\Gamma_2$ are constructed in an analogous way. 

To keep this article to a reasonable length, we only state the results for some cofinite Hecke triangle groups. For the others, analogous results are true. Recall that for $g\in\wt\Gamma_\lambda$ we have $N(g) = N(g^2)^{1/2}$.

\begin{thm}\label{billTO}
Let $\Gamma_\lambda$ be a cofinite Hecke triangle group with $\lambda = 2\cos\frac{\pi}{q}$, $q$ odd. Set $m \sceq \frac{q+1}{2}$.
\begin{enumerate}[{\rm (i)}]
\item The transfer operator associated to the slow discrete dynamical system for $\wt\Gamma_\lambda$ is
\[
 \mc L^{\text{slow}}_{s,\chi} = \sum_{k=m}^{q-1} \alpha_s(g_k) + \alpha_s(Qg_k),
\]
acting on $\Fct( (0,1); V)$.
\item\label{sonder} Suppose that $V=\C$ and $\chi$ acts trivially on $\Gamma_\lambda$.
\begin{enumerate}[{\rm (a)}]
\item If $\chi(Q) = 1$ and $\Rea s \in (0,1)$, then the even Maass cusp forms for $\Gamma_\lambda$ are isomorphic to the eigenfunctions with eigenvalue $1$ of $\mc L^{\text{slow}}_{s,\chi}$ that satisfy the regularity conditions from Theorem~\ref{MP_slow}.
\item If $\chi(Q) = -1$ and $\Rea s \in (0,1)$, then the odd Maass cusp forms for $\Gamma_\lambda$ are isomorphic to the eigenfunctions with eigenvalue $1$ of $\mc L^{\text{slow}}_{s,\chi}$ that satisfy the regularity conditions from Theorem~\ref{MP_slow}. 
\end{enumerate}
\item The transfer operator associated to the fast discrete dynamical system for $\wt\Gamma_\lambda$ is
\[
 \mc L^{\text{fast}}_{s,\chi} = 
\begin{pmatrix}
\sum\limits_{n\in\N} \alpha_s(Qg_{q-1}^n) & \sum\limits_{k=m}^{q-2} \alpha_s(g_k) + \alpha_s(Qg_k)
\\
\sum\limits_{n\in\N} \alpha_s(g_{q-1}^n) + \alpha_s(Qg_{q-1}^n) &  \sum\limits_{k=m}^{q-2} \alpha_s(g_k) + \alpha_s(Qg_k)
\end{pmatrix},
\]
acting on $B(\mc T^{-1}.\mc E_{q-1};V)\oplus B(\mc T^{-1}.\mc E_r;V)$. For $\Rea s > \frac12$, they define nuclear operators of order $0$, and $s\mapsto \mc L^{\text{fast}}_{s,\chi}$ extends meromorphically to all of $\C$ with poles located as in Theorem~\ref{TO1} of order at most $2\sd(\chi)$.
\item For $\Rea s > 1$ we have
\[
 \det(1-\mc L^{\text{fast}}_{s,\chi}) = \prod_{[g]\in [\wt\Gamma_\lambda]_p} \prod_{k=0}^\infty \det\left( 1 - \chi(g) \det g^k N(g)^{-(s+k)}\right),
\]
which is a dynamical zeta function.
\end{enumerate}
\end{thm}

\begin{proof}
These statements (and also the corresponding statements for the other Hecke triangle groups) are easily adapted from \cite{Pohl_hecke_spectral, Pohl_hecke_infinite}.
\end{proof}

In Theorem~\ref{billTO}\eqref{sonder} we see that this specific choice of representation mimicks Dirichlet ($\chi(Q) = 1$) respectively Neumann ($\chi(Q) = -1$) boundary conditions. The same interpretation also holds for the fast transfer operators and their dynamical zeta functions, for which we refer to \cite{Pohl_hecke_spectral}. We remark that these results allow us to reformulate the Phillips--Sarnak conjecture on the non-existence of even Maass cusp forms in terms of non-existence of non-trivial highly regular $1$-eigenfunctions of $\mc L^{\text{slow}}_{s,\chi}$ and $\mc L^{\text{fast}}_{s,\chi}$.

It would be interesting to see if there is a similar interpretation for more general representations, in particular if $\dim V > 1$.

\bibliography{Pohl_ap_bib}
\bibliographystyle{amsplain}
\end{document}

%% file: Pohl_praeambel.tex
\newcommand{\apref}[3]{\hyperref[#2]{#1\ref*{#2}#3}}

\usepackage{enumerate}
\usepackage[latin1]{inputenc}
\usepackage{dsfont}
\usepackage{amssymb,amsthm,amsmath}

\input{xy}
\xyoption{all}

\usepackage{mathrsfs}

\theoremstyle{plain}
\newtheorem{prop}{Proposition}[section]
\newtheorem{lemma}[prop]{Lemma}

\newtheorem{thm}[prop]{Theorem}

\newtheorem{cor}[prop]{Corollary}

\theoremstyle{definition}

\newtheorem{conj}[prop]{Conjecture}

\theoremstyle{remark}

%% file: Pohl_ap_makros.tex
\DeclareMathOperator{\PSO}{PSO}
\DeclareMathOperator{\Det}{det}
\DeclareMathOperator{\sd}{sd}
\DeclareMathOperator{\Primlength}{PL}
\DeclareMathOperator{\weight}{wt}
\DeclareMathOperator{\base}{base}

\DeclareMathOperator{\PSL}{PSL}
\DeclareMathOperator{\PGL}{PGL}

\DeclareMathOperator{\diag}{diag}

\DeclareMathOperator{\Tr}{Tr}
\DeclareMathOperator{\tr}{tr}

\DeclareMathOperator{\Ima}{Im}
\DeclareMathOperator{\Rea}{Re}

\DeclareMathOperator{\Stab}{Stab}

\DeclareMathOperator{\bd}{bd}

\newcommand{\st}{\text{st}}

\newcommand\N{\mathbb{N}}
\newcommand\Q{\mathbb{Q}}
\newcommand\R{\mathbb{R}}
\newcommand\Z{\mathbb{Z}}
\newcommand\C{\mathbb{C}}

\newcommand{\h}{\mathbb{H}}

\newcommand{\mc}[1]{\mathcal #1}

\newcommand{\wt}{\widetilde}
\newcommand{\wh}{\widehat}

\DeclareMathOperator{\dvol}{dvol}

\DeclareMathOperator{\id}{id}

\DeclareMathOperator{\Fct}{Fct}

\newcommand{\sceq}{\mathrel{\mathop:}=}

\newcommand{\bmat}[4]{\begin{bmatrix} #1&#2\\#3&#4\end{bmatrix}}

\newcommand{\textbmat}[4]{\left[\begin{smallmatrix} #1&#2 \\ #3&#4
\end{smallmatrix}\right]}
